\documentclass[12pt]{amsart}

\usepackage{enumerate}

\usepackage[foot]{amsaddr}
\usepackage[english]{babel}

\allowdisplaybreaks

\usepackage[
pdftex,hyperfootnotes]{hyperref}
\hypersetup{
	colorlinks=true,
	linkcolor=NavyBlue, 
	urlcolor=RoyalPurple,
	citecolor=OliveGreen,
	pdftitle={Geronimus Perturbations for Mixed Multiple Orthogonal Polynomials},
	bookmarks=true,
}
\usepackage[usenames,dvipsnames,svgnames,table,x11names]{xcolor}
\usepackage{pgfplots}
\usepackage{tikz}
\usepackage{tikz-3dplot}
\usetikzlibrary{automata,quotes, chains,matrix,calc,shadows,shapes.callouts,shapes.geometric,shapes.misc,positioning,patterns,decorations.shapes,
	decorations.pathmorphing,decorations.markings,decorations.fractals,decorations.pathreplacing,shadings,fadings,arrows.meta,bending}

\usepackage{nicematrix,drawmatrix}
\usepackage[utf8]{inputenc}

\usepackage{comment}

\usepackage[textwidth=17.5cm,textheight=22.275cm,
height=
24.275cm,
width=18cm]{geometry}

\usepackage{amssymb,latexsym,amsmath,amsthm,bm}
\usepackage{mathrsfs}
\usepackage{mathtools,arydshln,mathdots}
\mathtoolsset{showonlyrefs}
\usepackage{tcolorbox}

\usepackage[table]{xcolor}

\usepackage{Baskervaldx}
\usepackage[]{newtxmath}

\usepackage[all]{xy}
\hypersetup{
	colorlinks=true,
	linkcolor=blue}
\date{}
\providecommand{\abs}[1]{\lvert#1\rvert}

\theoremstyle{plain}

\newtheorem{Theorem}{Theorem}[section]
\newtheorem{Corollary}[Theorem]{Corollary}
\newtheorem{Lemma}[Theorem]{Lemma}
\newtheorem{Proposition}[Theorem]{Proposition}
\newtheorem{Definition}[Theorem]{Definition}
\newtheorem{Remark}[Theorem]{Remark}

\newcommand{\ii}{\operatorname{i}}

\renewcommand{\d}{\operatorname{d}}

\newcommand{\sgn}{\operatorname{sgn}}

\newcommand{\diag}{\operatorname{diag}}

\DeclarePairedDelimiter{\ceil}{\lceil}{\rceil}

\newcommand{\C}{\mathbb{C}}

\newcommand{\N}{\mathbb{N}}

\begin{document}
	
		\title[Geronimus Perturbations for Mixed Multiple Orthogonality]{General Geronimus Perturbations for  \\Mixed Multiple Orthogonal Polynomials}

	\author[M Mañas]{Manuel Mañas$^{1}$}
	
	\author[M Rojas]{Miguel Rojas$^{2}$}
	\address{Departamento de Física Teórica, Universidad Complutense de Madrid, Plaza Ciencias 1, 28040-Madrid, Spain}
	\email{$^{1}$manuel.manas@ucm.es}
	\email{$^{2}$migroj01@ucm.es}


		\begin{abstract}
General Geronimus transformations, defined by regular matrix polynomials that are neither required to be monic nor restricted by the rank of their leading coefficients, are applied through both right and left multiplication to a rectangular matrix of measures associated with mixed multiple orthogonal polynomials. These transformations produce Christoffel-type formulas that establish relationships between the perturbed and original polynomials. Moreover, it is proven that the existence of Geronimus-perturbed orthogonality is equivalent to the non-cancellation of certain \(\tau\)-determinants. The effect of these transformations on the Markov--Stieltjes matrix functions is also determined. As a case study, we examine the Jacobi--Piñeiro orthogonal polynomials with three weights.
	\end{abstract}
	
	\maketitle
\tableofcontents
        \section{Introduction}
Multiple orthogonal polynomials (MOPs) form a general class of polynomials associated with multiple weight functions, unlike classical orthogonal polynomials tied to a single weight. MOPs play essential roles in numerical analysis, approximation theory, and mathematical physics, owing to their capacity to address complex problems involving simultaneous orthogonality conditions.

MOPs are historically linked to Hermite–Padé approximations and constructive function theory. For a detailed introduction, Nikishin and Sorokin's book \cite{nikishin_sorokin} and Van Assche's chapter in \cite[Ch. 23]{Ismail} provide excellent starting points. Connections between MOPs and integrable systems are explored in \cite{afm}, while \cite{andrei_walter} offers a more accessible introduction. Studies on the asymptotic behavior of their zeros are presented in \cite{Aptekarev_Kaliaguine_Lopez}, and Gauss–Borel perspectives are discussed in \cite{afm}. Applications in random matrix theory are detailed in \cite{Bleher_Kuijlaars}.

Mixed-type MOPs, along with the corresponding Riemann–Hilbert problems, find applications in contexts such as Brownian bridges and non-intersecting Brownian motions \cite{Evi_Arno}, and in multicomponent Toda systems \cite{adler,afm}. In number theory, they were used in Apéry’s proof of the irrationality of $\zeta(3)$ \cite{Apery} and in proving the irrationality of certain values of the $\zeta$-function at odd integers \cite{Ball_Rivoal}.

The logarithmic and ratio asymptotics of linear forms constructed from Nikishin systems, which satisfy orthogonality conditions with respect to a second Nikishin system, were explored in \cite{ulises}. Further research in \cite{ulises2} examined a broad class of mixed-type MOPs and the properties of their corresponding zeros.

Recent work highlights the role of mixed-type MOPs in the Favard spectral description of banded semi-infinite matrices, explored in studies like \cite{aim, phys-scrip, BTP, Contemporary}, with additional insights in \cite{laa}. These polynomials are also significant in the analysis of Markov chains and random walks that go beyond birth and death processes, as shown in \cite{CRM,finite,hypergeometric,JP}.

In 1858, the German mathematician Elwin Christoffel \cite{christoffel} studied Gaussian quadrature rules, aiming to find explicit formulas relating orthogonal polynomial sequences under different measures. Specifically, he investigated the Lebesgue measure $\mathrm{d}\mu = \mathrm{d}x$ and a modified measure $\d\hat{\mu}(x) = p(x) \d\mu(x)$, where $p(x) = (x - q_1) \cdots (x - q_N)$ is a polynomial with roots outside the support of $\mathrm{d}\mu$. Christoffel sought to understand the distribution of zeros, or nodes, in such quadrature rules \cite{Uva}. The resulting Christoffel formula is discussed in classical textbooks on orthogonal polynomials, such as \cite{Chi,Sze,Gaut}. A fresher overview of Christoffel and Geronimus transformations can be found in \cite{manas}.

These transformations extend beyond measures to involve linear functionals \cite{ahi,Chi,Sze}. For a moment functional $u$, its canonical Christoffel transformation is defined as $\hat{u} = (x - a)u$, where $a \in \mathbb{R}$ \cite{Bue1, Chi, Yoon}. Conversely, the right inverse of a Christoffel transformation is known as the Geronimus transformation. For a moment functional $u$, its Geronimus transformation yields a new moment functional $\check{u}$ that satisfies $(x - a)\check{u} = u$. The Geronimus transformation involves a free parameter \cite{Geronimus,Maro}, while a general Christoffel transformation’s right inverse is called a multiple Geronimus transformation \cite{DereM}.

These perturbations, including Christoffel and Geronimus, belong to the broader category of Darboux transformations, initially introduced in the context of integrable systems \cite{matveev}. Gaston Darboux formally addressed these transformations in 1878 while studying Sturm–Liouville theory \cite{darboux2,moutard}. The factorization of Jacobi matrices akin to these transformations has been investigated in the context of orthogonal polynomials on the real line \cite{Bue1,Yoon} and is essential in the study of bispectral problems \cite{gru,gru2}.

The canonical Christoffel transformations are closely related to the $LU$ factorization (or its flipped version, $UL$ factorization) of the Jacobi matrix, which results from the three-term recurrence relation of a sequence of monic polynomials orthogonal with respect to a nontrivial probability measure $\mu$. This factorization allows the derivation of another Jacobi matrix $\hat{J}$, along with its sequence of monic polynomials $\{\hat{P}_{n}(x)\}_{n=0}^{\infty}$, which are orthogonal with respect to the canonical Christoffel-transformed measure $\mu$.

In the context of a moment functional $u$, the Markov–Stieltjes function $S(x)$ is fundamental in orthogonal polynomial theory. It is closely related to the measure corresponding to $u$ and its rational Padé approximations \cite{Bre,Karl}. For the Christoffel transformation $\hat{u}$ of the moment functional $u$, the Stieltjes function takes the form $\hat{S}(x) = (x - a)S(x) - u_0$, which is a specific case of spectral linear transformations \cite{Zhe}.

The first author and collaborators have previously studied Christoffel and Geronimus transformations in the context of matrix polynomials. The research began with the application of Christoffel transformations to monic matrix orthogonal polynomials, as detailed in \cite{AAGMM}. Subsequently, Geronimus transformations in the matrix case were investigated in \cite{AGMM}, where spectral techniques were introduced for monic perturbations, while the non-monic case was addressed without employing spectral methods. In \cite{AGMM2}, the Geronimus–Uvarov framework was explored, with a particular focus on its connections to non-Abelian Toda lattices.

Moreover, in \cite{bfm}, the Christoffel and Geronimus perturbations of two weights for non-mixed MOPs was studied. It was  presented connection formulas between type II multiple orthogonal polynomials, type I linear forms, and vector Stieltjes–Markov functions. The perturbation matrix polynomials in this case were not necessarily monic but belonged to a restricted class.

In \cite{Manas_Rojas}, we explored general Christoffel transformations for mixed multiple orthogonal polynomials. By applying regular matrix polynomials—neither required to be monic nor constrained by the rank of their leading coefficients—through both right and left multiplication to a rectangular matrix of measures, we derived Christoffel-type formulas that relate the perturbed polynomials to the original ones. Using the divisibility theorem for matrix polynomials, we also established a criterion for the existence of perturbed orthogonality, characterized by the non-cancellation of specific $\tau$-determinants.

In this paper, we extend our investigation to examine general Geronimus transformations. These transformations, defined by regular matrix polynomials that are neither required to be monic nor limited by the rank of their leading coefficients, are applied via both right and left multiplication to a rectangular matrix of measures associated with mixed multiple orthogonal polynomials. The resulting transformations lead to Christoffel-type formulas that connect the perturbed polynomials to the original ones. Additionally, we prove that the existence of Geronimus-perturbed orthogonality is equivalent the non-cancellation of certain $\tau$-determinants. 
The impact of these transformations on the Markov--Stieltjes matrix functions is also analyzed. As a case study, we focus on the Jacobi--Piñeiro orthogonal polynomials with three weights.

\subsection{Mixed Multiple Orthogonal Polynomials on the Step Line}

For  $p,q\in\N$, consider the setup where a rectangular matrix of measures is defined:
\[
\d	\mu=\begin{bNiceMatrix}
	\d\mu_{1,1}&\Cdots &\d\mu_{1,p}\\
	\Vdots & & \Vdots\\
	\d	\mu_{q,1}&\Cdots &\d\mu_{q,p}
\end{bNiceMatrix},
\]
with each \(\mu_{b,a}\) being a measure supported on the interval \(\Delta_{a,b} \subseteq \mathbb{R}\). The support of the matrix of measures is said to be $\Delta\coloneq \cup_{a,b}\Delta_{a,b}$. In this paper we use the notation $\N\coloneq \{1,2,\dots\}$ and $N_0\coloneq \{0,1,2,\dots\}$.

For \(r \in \mathbb{N}\), we define the matrix of monomials:
\[
X_{[r]}(x) = \begin{bNiceMatrix}
	I_r \\
	xI_r \\
	x^2I_r \\
	\Vdots
\end{bNiceMatrix},
\]
and the moment matrix as:
\[
\mathscr{M}\coloneq \int_{\Delta} X_{[q]}(x) \d\mu(x) X_{[p]}^\top(x).
\]

If all leading principal submatrices \(\mathscr{M}^{[k]}\) are invertible, a $LU$ factorization exists:
\[
\mathscr{M} = \mathscr{L}^{-1} \mathscr{U}^{-1},
\]
where \(\mathscr{L}\) and \(\mathscr{U}\) are nonsingular lower and upper triangular matrices. This factorization is  unique up  to
\(
\mathscr{L} \to \mathscr{d}^{-1}\mathscr{L}, \quad \mathscr{U} \to \mathscr{U}\mathscr{d},
\)
for \(\mathscr{d}\)  any nonsingular diagonal matrix.

Two normalizations are of importance:
\begin{itemize}
	\item  Left normalization sets \(\mathscr{L}\) as unitriangular, resulting in triangular matrices \(S\) and \(\mathscr{U}_L\).
\item Right normalization sets \(\mathscr{U}\) as unitriangular, yielding matrices \(\mathscr{L}_R\) and \(\bar S^\top\).
\end{itemize}

The $LU$ factorization can be rewritten as:
\[
\mathscr{M} = S^{-1} H\bar S^{-\top},
\]
where \(S\), \(\bar S\), and \(H\) are lower unitriangular and  invertible diagonal matrices, respectively. 

Next, we define matrix polynomials associated with this factorization:
\[
\begin{aligned}
	B(x) &= \mathscr{L} X_{[q]}(x), & A(x) &= X_{[p]}^\top(x) \mathscr{U},
\end{aligned}
\]
where \(B(x)\) is monic, i.e.,
\[
\begin{aligned}
	B(x) &= S X_{[q]}(x), & A(x) &= X_{[p]}^\top(x) \bar S^\top H^{-1}.
\end{aligned}
\]
The polynomial entries of \(B\) and \(A\) are:
\[
\begin{aligned}
	B &= \begin{bNiceMatrix}
	B^{(1)}_0 & \Cdots & B^{(q)}_0 \\[2pt] B^{(1)}_1 & \Cdots & B^{(q)}_1 \\[2pt] B^{(1)}_2 & \Cdots & B^{(q)}_2 \\ \Vdots & & \Vdots
\end{bNiceMatrix}, &
A &= \left[\begin{NiceMatrix}
	A^{(1)}_0 & A^{(1)}_1 & A^{(1)}_2 & \Cdots \\ \Vdots & \Vdots & \Vdots & \\ A^{(p)}_0 & A^{(p)}_1 & A^{(p)}_2 & \Cdots
\end{NiceMatrix}\right].
\end{aligned}
\]
We will make use of the following notation,  
\[
\begin{aligned}
    B_n(x) & \coloneq \begin{bNiceMatrix}
        B^{(1)}_n(x) & \Cdots & B^{(q)}_n(x)
    \end{bNiceMatrix}, & A_n(x) & \coloneq \begin{bNiceMatrix}
        A_n^{(1)}(x) \\ \Vdots \\ A_n^{(p)}(x)
    \end{bNiceMatrix}.
\end{aligned}
\]
The polynomials $B(x)$ and $A(x)$ satisfy the biorthogonality condition:
\[
\begin{aligned}
	\int_{\Delta} B(x) \, \d\mu(x) \, A(x) &= I, &
\int_{\Delta} \sum_{b=1}^q \sum_{a=1}^p B^{(b)}_n(x) \d\mu_{b,a}(x) A^{(a)}_m(x) &= \delta_{n,m}.
\end{aligned}
\]

The $LU$ factorization yields:
\[
\begin{aligned}
	\int_\Delta B(x) \, \d\mu(x) X_{[p]}^\top(x) &= \mathscr{U}^{-1}, & \int_\Delta X_{[q]}(x) \, \d\mu(x) A(x)& = \mathscr{L}^{-1},
\end{aligned}
\]
which gives rise to diagonal orthogonality relations:
\[
\begin{aligned}
	\int_\Delta x^l \sum_{a=1}^p \d\mu_{b,a}(x) A_n^{(a)}(x) &= 0, & b &\in\{1,\dots,q\},& l &\in\left\{0,\dots, \left\lceil\frac{n-b+2}{q}\right\rceil-1\right\} ,\\
\int_\Delta \sum_{b=1}^q B_n^{(b)}(x) \d\mu_{b,a}(x) x^l &= 0,  & a &\in\{1,\dots,p,\}, &l &\in\left\{0,\dots,\left\lceil\frac{n-a+2}{p}\right\rceil-1
\right\} .
\end{aligned}
\]

Consequently, the existence of the $LU$ factorization is equivalent to the existence of the orthogonality.

For \(r \in \mathbb{N}_0\), the shift matrix is:
\[
\Lambda_{[r]} \coloneq 
\left[\begin{NiceMatrix}
	0_r & I_r & 0_r & \Cdots \\ 0_r & 0_r & I_r & \Ddots \\ 0_r & 0_r & 0_r & \Ddots \\ \Vdots & \Ddots[shorten-end=5pt] & \Ddots[shorten-end=10pt] & \Ddots[shorten-end=15pt]
\end{NiceMatrix}\right],
\]
with \(\Lambda_{[1]} = \Lambda\) and \(\Lambda_{[r]} = \Lambda^r\). These matrices satisfy:
\[
\Lambda_{[r]}X_{[r]}(x) = x X_{[r]}(x).
\]
The moment matrix \(\mathscr{M}\) obeys the symmetry:
\[
\Lambda_{[q]}\mathscr{M} = \mathscr{M}\Lambda_{[p]}^\top,
\]
leading to:
\[
T = \mathscr{L} \Lambda_{[q]} \mathscr{L}^{-1} = \mathscr{U}^{-1} \Lambda_{[p]}^\top \mathscr{U},
\]
implying that \(T\) is a \((p,q)\)-banded matrix with \(p\) subdiagonals and \(q\) superdiagonals. The recurrence relations for mixed multiple orthogonal polynomials are:
\begin{equation}\label{eq:recursion}
\begin{aligned}
		T B(x) &= x B(x), & A(x) T &= x A(x),
\end{aligned}
\end{equation}
where \(T\) is the recurrence matrix. The leading principal submatrices take the form
\[
		T^{[n]}=
\begin{bNiceMatrix}[columns-width = 1.5cm]
	T_{0,0} &\Cdots &&T_{0,q}& 0 & \Cdots &&0\\
	\Vdots& &&&\Ddots &\Ddots& &\Vdots\\
	T_{p,0}&&&&&&&\\
	0&\Ddots&&&&&&0\\
	\Vdots&\Ddots&&&&&&T_{n-q,m}\\
	&&&&&&&\Vdots\\
	&&&&&&&\\[6pt]
	0&\Cdots&&&0&T_{n,n-p}&\Cdots&T_{n,n}
\end{bNiceMatrix}.
\]
So that  we have the inductive limit $T=\lim_{n\to\infty}  T^{[n]}$.
\subsection{Christoffel--Darboux Kernels}

We now turn our attention to key elements essential for the constructions in this paper: the kernel polynomials.
\begin{Definition}
	Let us introduce the Christoffel--Darboux (CD) kernel  $p\times q$ matrix polynomial, $K^{[n]}(x,y)$:
	\begin{align} \label{CDkernel}
		K^{[n]}(x,y) & \coloneq A^{[n]}(x)B^{[n]}(y) = \begin{bNiceMatrix}
			A_0^{(1)}(x) & \Cdots & A_{n}^{(1)}(x) \\
			\Vdots & & \Vdots \\
			A_0^{(p)}(x) & \Cdots & A_{n}^{(p)}(x)
		\end{bNiceMatrix} \begin{bNiceMatrix}
			B_0^{(1)}(y) & \Cdots & B_0^{(q)}(y) \\
			\Vdots & & \Vdots \\
			B_{n}^{(1)}(y) & \Cdots & B_{n}^{(q)}(y) 
		\end{bNiceMatrix},\\
		K^{[n]}_{a,b}(x,y) & = \sum_{i=0}^{n}A^{(a)}_i(x)B_i^{(b)}(y).
	\end{align}
\end{Definition}

This kernel polynomial satisfies several interesting properties, similar to the  properties of the kernel polynomial in the scalar case.
\begin{Proposition} \label{ProyectorK}
	For any given matrix polynomial of degree $N$, i.e., 
	\[P(x) = \sum_{i=0}^N P_i x^i, \quad P_i \in \mathbb{C}^{p\times p},\]
	the following projection property holds: 
	\begin{equation} \label{Eq ProyectorK}
		\begin{aligned}
			\int_{\Delta}K^{[n]}(x,t)\d \mu(t) P(t) &= P(x), &  n &\geq Np + p -1.
		\end{aligned}
	\end{equation}
\end{Proposition}
\begin{proof}
	Using the orthogonality relations, for $n \geq Np + p-1$, it is straightforward to verify that: 
	\begin{equation*}
		\int_\Delta \begin{bNiceMatrix}
			B_n^{(1)}(t) & \Cdots & B_n^{(q)}(t)
		\end{bNiceMatrix} \d \mu(t) P(t) = 0.
	\end{equation*}
	Consequently, the result is independent of $n$: 
	\begin{equation*}
		A^{[n]}(x)\int_{\Delta}B^{[n]}(t) \d \mu(t) P(t) = A^{[Np+p]}(x)\int_{\Delta}B^{[Np+p]}(t) \d \mu(t) P(t) = A(x)\int_{\Delta}B(t) \d \mu(t) P(t).
	\end{equation*}
	Now, we prove the projection property for $n \rightarrow \infty$, 
	\begin{align*}
		A(x)\int_{\Delta}B(t) \d \mu(t) P(t) = X_{[p]}^\top(x) \Bar{S}^\top H^{-1}S \int_\Delta X_{[q]}(t) \d \mu(t) X_{[p]}^\top(t) \begin{bNiceMatrix}
			P_0 \\ \Vdots \\ P_N \\ 0_p \\ \Vdots
		\end{bNiceMatrix}= X_{[p]}^\top(x)\begin{bNiceMatrix}
			P_0 \\ \Vdots \\ P_N \\ 0_p \\ \Vdots
		\end{bNiceMatrix} = P(x).
	\end{align*}
	To achieve the stated result, we have used the definition of the moment matrix and its Gauss--Borel factorization.
\end{proof}

Christoffel–Darboux formulas of the type described here were discussed in \cite{Evi_Arno,AM,afm} for mixed multiple orthogonality. For another type of Christoffel–Darboux formulas for mixed multiple orthogonality, see \cite{aim}.

For the banded recurrence matrix \(T\) corresponding to a given \(n\), we consider two additional matrices:
	\begin{align*}
	\mathscr T^{[n,>n]}&\coloneq\begin{bNiceMatrix}
		T_{n+1-q,n+1}& 0 &\Cdots &0\\
		\Vdots &\Ddots &\Ddots&\Vdots\\
		&&&0\\
		T_{n,n+1}&\Cdots&&T_{n,n+q}
	\end{bNiceMatrix}, &
	\mathscr T^{[>n,n]}&\coloneq\begin{bNiceMatrix}
		T_{n+1,n+1-p}&  \Cdots& &T_{n+1,n}\\
		0&\Ddots[shorten-start=-1pt,shorten-end=-25pt]&&\\\\[-16pt]
		\Vdots & \Ddots[shorten-start=-25pt]&&\Vdots\\	0&\Cdots&0&T_{n+p,n}
	\end{bNiceMatrix}, 
\end{align*}
where $\mathscr	T^{[n,>]}$  and $\mathscr	T^{[>,n]}$ are   $q\times q$ and $p\times p$ matrices, and
\[
\begin{aligned}
	\mathscr{A}^{[n]} &\coloneq \begin{bNiceMatrix}
		A^{(1)}_{n+1-q} & \Cdots & A^{(1)}_{n} \\
		\Vdots & & \Vdots \\
		A^{(p)}_{n+1-q} & \Cdots & A^{(p)}_{n}
	\end{bNiceMatrix}, &
	\mathscr{A}^{[>n]} &\coloneq \begin{bNiceMatrix}
		A^{(1)}_{n+1} & \Cdots & A^{(1)}_{n+p} \\
		\Vdots & & \Vdots \\
		A^{(p)}_{n+1} & \Cdots & A^{(p)}_{n+p}
	\end{bNiceMatrix}, \\
	\mathscr{B}^{[n]} &\coloneq \begin{bNiceMatrix}
		B^{(1)}_{n+1-p} & \Cdots & B^{(q)}_{n+1-p} \\
		\Vdots & & \Vdots \\
		B^{(1)}_{n} & \Cdots & B^{(q)}_{n}
	\end{bNiceMatrix}, &
	\mathscr{B}^{[>n]} &\coloneq \begin{bNiceMatrix}
		B^{(1)}_{n+1} & \Cdots & B^{(q)}_{n+1} \\
		\Vdots & & \Vdots \\
		B^{(1)}_{n+q} & \Cdots & B^{(q)}_{n+q}
	\end{bNiceMatrix},
\end{aligned}
\]
where \(\mathscr{A}^{[n]}\), \(\mathscr{A}^{[>n]}\), \(\mathscr{B}^{[n]}\), and \(\mathscr{B}^{[>n]}\) are \(p \times q\), \(p \times p\), \(p\times q\), and \(q \times q\) matrix polynomials, respectively.

We can express the matrix \(T\) as:
\[
T = \left[\begin{NiceArray}{c|c}
	T^{[n]} & T^{[n,>n]} \\ \hline
	T^{[>n,n]} &
\end{NiceArray}\right],
\]
where \(T^{[n,>n]}\) and \(T^{[>n,n]}\) are \((n+1) \times \infty\) and \(\infty \times (n+1)\) matrices, respectively. These blocks are given by:
\[
\begin{aligned}
	T^{[>n,n]} &= \left[\begin{NiceArray}{c|c}
		0_{p \times (n+1-p)} & \mathscr{T}^{[>n,n]} \\
		\hline
		0_{\infty \times (n+1-p)} & 0_{\infty \times \infty}
	\end{NiceArray}\right], &
	T^{[n,>n]} &= \left[\begin{NiceArray}{c|c}
		0_{(n+1-q) \times q} & 0_{\infty \times \infty} \\
		\hline
		\mathscr{T}^{[n,>n]} & 0_{(n+1-q) \times \infty}
	\end{NiceArray}\right].
\end{aligned}
\]

Similarly, we introduce the notation:
\[
\begin{aligned}
	A &= \begin{bNiceArray}{c|c}
		A^{[n]} & A^{[>n]}
	\end{bNiceArray}, &
	B &= \left[\begin{NiceArray}{c}
		B^{[n]} \\ \hline B^{[>n]}
	\end{NiceArray}\right].
\end{aligned}
\]

\begin{Theorem}\label{KcomoT}
	The following matrix Christoffel–Darboux formula holds:
	\[
	(x-y)K^{[n]}(x,y) = \mathscr{A}^{[>n]}(x)\mathscr{T}^{[>n,n]}\mathscr{B}^{[n]}(y) - \mathscr{A}^{[n]}(x)\mathscr{T}^{[n,>n]}\mathscr{B}^{[>n]}(y).
	\]
\end{Theorem}

\begin{proof}
	From \eqref{eq:recursion}, we have \((A(x)T)^{[n]} = xA^{[n]}(x)\) and \((TB(y))^{[n]} = y B^{[n]}(y)\). Expanding these:
	\[
	\begin{aligned}
		xA^{[n]}(x) &= (A(x)T)^{[n]} = A^{[n]}(x)T^{[n]} + A^{[>n]}T^{[>n,n]}, \\
		yB^{[n]}(y) &= (TB(y))^{[n]} = T^{[n]}B^{[n]}(y) + T^{[n,>n]}B^{[>n]}(y).
	\end{aligned}
	\]
	Thus, we obtain:
	\[
	\begin{aligned}
		(x-y)K^{[n]}(x,y) &= A^{[>n]}(x)T^{[>n,n]}B^{[n]}(y) - A^{[n]}(x)T^{[n,>n]}B^{[>n]}(y) \\
		&= \mathscr{A}^{[>n]}(x)\mathscr{T}^{[>n,n]}\mathscr{B}^{[n]}(y) - \mathscr{A}^{[n]}(x)\mathscr{T}^{[n,>n]}\mathscr{B}^{[>n]}(y).
	\end{aligned}
	\]
\end{proof}

\subsection{Cauchy Transforms}
Another crucial set of objects in this construction are the Cauchy transforms of the matrix polynomials under consideration.
\begin{Definition}
	The Cauchy transforms of the orthogonal polynomials $A(x)$ and $B(x)$ are defined as follows:
	\[    \begin{aligned}
		C(z) & \coloneq \int_\Delta \frac{\d \mu(x)}{z-x}A(x), & D(z) & \coloneq \int_\Delta B(x) \frac{\d \mu(x)}{z-x}.
	\end{aligned}\]
	
\end{Definition}

\begin{Remark}
	The entries of the Cauchy transform matrices are holomorphic in  $\C\setminus \Delta$.
\end{Remark}
\begin{Proposition}
	$C(z)$ and $D(z)$ can also be expressed in terms of the Gauss--Borel matrices, $S$ and $\Bar{S}$:
	\[ \begin{aligned}
		C(z) & = \frac{1}{z} X_{[q]}^\top(z^{-1}) S^{-1}, & D(z) & = \frac{1}{z} H \Bar{S}^{-\top} X_{[p]}(z^{-1}),
	\end{aligned}\]
	whenever $\abs{z} > \mathrm{sup}\{ \abs{x}: x \in \Delta \}$.
\end{Proposition}
\begin{proof}
	We have
	\begin{align*}
		C(z) & = \int_\Delta \frac{\d \mu(x)}{z-x}A(x) = \frac{1}{z} \int_\Delta \frac{\d \mu(x)}{1-\frac{x}{z}}A(x) = \frac{1}{z}\int_\Delta \sum_{i=0}^\infty \left( \frac{x}{z} \right)^i \d \mu(x) A(x) \\[3pt]
		& = \frac{1}{z} X_{[q]}^\top(z^{-1})\int_\Delta X_{[q]}(x) \d \mu(x) X^\top_{[p]}(x) \Bar{S}^\top H^{-1} = \frac{1}{z} X_{[q]}(z^{-1}) S^{-1}.
	\end{align*}
	The expansion of $(1-\frac{x}{z})^{-1}$ in a power series is valid since $\abs{z} > \mathrm{sup}\{ \abs{x}: x \in \Delta \}$. A similar proof can be given for $D(z)$. 
\end{proof}
Let us now introduce two additional families of CD kernels.
\begin{Definition}
	The mixed Christoffel--Darboux kernels are defined as follows:
	\begin{align*}
		K^{[n]}_{C}(x,y) & \coloneq C^{[n]}(x)B^{[n]}(y) = \int_\Delta \frac{\d \mu(t)}{x-t}K^{[n]}(t,y), \\
		K^{[n]}_{D}(x,y) & \coloneq A^{[n]}(x)D^{[n]}(y) = \int_\Delta K^{[n]}(x,t)\frac{\d \mu(t)}{y-t}.
	\end{align*}
\end{Definition}

We now derive Christoffel–Darboux formulas for the mixed kernels. To this end, we introduce
\[
\begin{aligned}
	\mathscr{C}^{[n]}(x) &\coloneq \int_{\Delta} \frac{\mu(t)}{x-t} \mathscr{A}^{[n]}(t), &
	\mathscr{C}^{[>n]}(x) &\coloneq \int_{\Delta} \frac{\mu(t)}{x-t} \mathscr{A}^{[>n]}(t), \\
	\mathscr{D}^{[n]}(y) &\coloneq \int_{\Delta} \mathscr{B}^{[n]}(t) \frac{\mu(t)}{y-t}, &
	\mathscr{D}^{[>n]}(y) &\coloneq \int_{\Delta} \mathscr{B}^{[>n]}(t) \frac{\mu(t)}{y-t}.
\end{aligned}
\]

\begin{Theorem}\label{KDcomoT}
	The mixed Christoffel–Darboux kernels satisfy the following Christoffel–Darboux-type formulas:
	\[
	\begin{aligned}
		(x-y) K^{[n]}_{C}(x,y) 
		&= \big(\mathscr{C}^{[>n]}(x) - \mathscr{C}^{[>n]}(y)\big) \mathscr{T}^{[>n,n]} \mathscr{B}^{[n]}(y) 
		- \big(\mathscr{C}^{[n]}(x) - \mathscr{C}^{[n]}(y)\big) \mathscr{T}^{[n,>n]} \mathscr{B}^{[>n]}(y), \\
		(x-y) K^{[n]}_{D}(x,y) 
		&= \mathscr{A}^{[n]}(x)  \mathscr{T}^{[n,>n]} \big(\mathscr{D}^{[>n]}(x) - \mathscr{D}^{[>n]}(y)\big) - \mathscr{A}^{[>n]}(x)  \mathscr{T}^{[>n,n]} \big(\mathscr{D}^{[n]}(x) - \mathscr{D}^{[n]}(y)\big).
	\end{aligned}
	\]
\end{Theorem}

\begin{proof}
	We prove only the first formula, as the second follows similarly.
	
	We start with
	\[
	K^{[n]}_{C}(x,y) = \int_\Delta \frac{\d \mu(t)}{x-t} K^{[n]}(t,y),
	\]
	and use the Christoffel–Darboux formula:
	\[
	K^{[n]}_{C}(x,y) = \int_\Delta \frac{\d \mu(t)}{(x-t)(t-y)} 
	\left(\mathscr{A}^{[>n]}(t) \mathscr{T}^{[>n,n]} \mathscr{B}^{[n]}(y) 
	- \mathscr{A}^{[n]}(t) \mathscr{T}^{[n,>n]} \mathscr{B}^{[>n]}(y)\right).
	\]
	
	Since
	\[
	\frac{1}{(x-t)(t-y)} = \frac{1}{x-y} \left( \frac{1}{x-t} - \frac{1}{y-t} \right),
	\]
	we find
	\[
	\begin{aligned}
		(x-y) K^{[n]}_{C}(x,y) &= \int_\Delta \d \mu(t) \left( \frac{1}{x-t} - \frac{1}{y-t} \right)
		\left(\mathscr{A}^{[>n]}(t) \mathscr{T}^{[>n,n]} \mathscr{B}^{[n]}(y) 
		- \mathscr{A}^{[n]}(t) \mathscr{T}^{[n,>n]} \mathscr{B}^{[>n]}(y)\right) \\
		&= \big(\mathscr{C}^{[>n]}(x) - \mathscr{C}^{[>n]}(y)\big) \mathscr{T}^{[>n,n]} \mathscr{B}^{[n]}(y) 
		- \big(\mathscr{C}^{[n]}(x) - \mathscr{C}^{[n]}(y)\big) \mathscr{T}^{[n,>n]} \mathscr{B}^{[>n]}(y).
	\end{aligned}
	\]
\end{proof}

\subsection{Canonical Set of Jordan Chains and Divisibility for Matrix Polynomials} \label{S:CSoJC}

Building on \cite{MatrixPoly}, we explore essential results concerning matrix polynomials, crucial for our further study. We focus on regular matrix polynomials:
\[
R(x) = \sum_{l=0}^N R_l x^l, \quad R_l \in \mathbb{C}^{p \times p},
\]
where \(\det R(x)\) is not identically zero. The degree of \(\det R(x)\) is:
\[
\deg \det R(x) = Np - r, \quad r \in \{0,\dots,Np-1\}.
\]
The eigenvalues of \(R(x)\) are the zeros of \(\det R(x)\).

\begin{Proposition}[Smith Form]\label{SmithForm}
	Any matrix polynomial can be represented in its Smith form as:
	\[
	R(x) = E(x)D(x)F(x),
	\]
	where \( E(x) \) and \( F(x) \) are matrices with constant determinants, and \( D(x) \) is a diagonal matrix polynomial. Explicitly, \( D(x) \) takes the form:
	\[
	D(x) = \diag \left(
	\prod_{i=1}^{M}(x-x_i)^{\kappa_{i,1}}, \prod_{i=1}^{M}(x-x_i)^{\kappa_{i,2}}, \dots , \prod_{i=1}^{M}(x-x_i)^{\kappa_{i,p}}\right),
	\]
	where \( x_i \) are the roots of \( \det R(x) \), and \( \kappa_{i,j} \) are the partial multiplicities. For a root \( x_i \) with multiplicity \( K_i \), the following relations hold:
	\[
	\begin{aligned}
		K_i &= \sum_{j=1}^{p} \kappa_{i,j}, & \sum_{i=1}^{M}\sum_{j=1}^p \kappa_{i,j} &= Np - r,
	\end{aligned}
	\]
	with \( M \) being the number of distinct roots and some partial multiplicities potentially zero.
\end{Proposition}
To simplify the notation, we consider a single eigenvalue \( x_0 \) with multiplicity \( K \).

\begin{Definition}[Jordan Chains]
	\begin{enumerate}
		\item A Jordan chain for \( R(x) \) at \( x_0 \in \mathbb{C} \) consists of \( L+1 \) vectors satisfying:
		\[
		\sum_{l=0}^i \frac{1}{l!}\boldsymbol{v}_{L-l}R^{(l)}(x_0) = 0, \quad i \in \{0, \dots ,L\}.
		\]
		\item A canonical set of Jordan chains of \( R(x) \) at \( x_0 \) consists of \( K_i \) vectors structured as:
		\[
		\{ \boldsymbol{v}_{1,0}, \boldsymbol{v}_{1,1}, \dots, \boldsymbol{v}_{1,\kappa_1-1}, \dots, \boldsymbol{v}_{s,0}, \dots, \boldsymbol{v}_{s,\kappa_s-1} \},
		\]
		where \( s \leq p \) and \( \sum_{i=1}^s \kappa_i = K \). Each subset \( \{ \boldsymbol{v}_{i,0}, \dots, \boldsymbol{v}_{i,\kappa_i-1} \} \) forms a Jordan chain of length \( \kappa_i \), with the vectors \( \boldsymbol{v}_{i,0} \) being linearly independent.
	\end{enumerate}
\end{Definition}

	\subsection{The Matrix Structure of the Polynomial Perturbation}
Let's  consider  matrix polynomials  
\[	\begin{aligned}
	R(x) & = \sum_{l=0}^N  R_{l}x^l, & R_{l} &  \in \mathbb{C}^{p\times p}, 
\end{aligned}\]
with  leading and sub-leading matrix coefficients of the form:
\begin{align} \tag{C1}\label{LeadingMatrixConditions}
	R_{N}  & = \begin{bNiceArray}{cw{c}{1cm}c||w{c}{2cm}c}
		\Block{3-2}{0_{(p-r)\times r} }	&	& \Block{3-3}{\left[ t_{N} \right]_{(p-r)\times (p-r)}} &&\\\\\\
		\Block{2-2}{	0_{r\times r} }& &\Block{2-3}{0_{r \times (p-r)}} &&\\\\
	\end{bNiceArray}, & R_{N-1}  & = \begin{bmatrix}
		\left[ R^1_{N-1} \right]_{(p-r)\times r} & \left[ R^2_{N-1} \right]_{(p-r)\times (p-r)} \\ \\
		\left[ t_{N-1} \right]_{r\times r} & \left[ R^4_{N-1} \right]_{r \times (p-r)}
	\end{bmatrix},
\end{align}
where $r$  take values in $\{0,\cdots,p-1\}$, and $\left[ t_{N} \right]_{(p-r)\times (p-r)}$ and $\left[ t_{N-1} \right]_{r\times r}$ are upper triangular matrices with nonzero determinant.

From this point onward, we will consider matrix polynomials in which the leading and sub-leading matrices, $\left[ t_{N} \right]_{(p-r)\times (p-r)}$ and $\left[ t_{N-1} \right]_{r\times r}$, are chosen to be the identity matrix:
\begin{align} \tag{C2}\label{CondicionesMatricesLideresFinales}
	R_{N}  & = \begin{bmatrix}
		0_{(p-r)\times r} & I_{(p-r)\times (p-r)} \\ \\
		0_{r\times r} & 0_{r \times (p-r)}
	\end{bmatrix}, & R_{N-1}  & = \begin{bmatrix}
		\left[ R^1_{N-1} \right]_{(p-r)\times r} & \left[ R^2_{N-1} \right]_{(p-r)\times (p-r)} \\ \\
		I_{r\times r} & \left[ R^4_{N-1} \right]_{r \times (p-r)}
	\end{bmatrix}.
\end{align}
It is straightforward to observe that any matrix polynomial of the form in condition \eqref{LeadingMatrixConditions} can be expressed as the product of $R(x)$ and another upper triangular matrix with a nonzero determinant. While we will focus on perturbations where the leading matrices satisfy the condition \eqref{CondicionesMatricesLideresFinales}, multiplying the weight matrix by a matrix with a nonzero determinant will preserve orthogonality, and the newly perturbed polynomials will be linear combinations of the original ones.

\begin{Proposition} \label{Prop2}
The determinant of a matrix polynomial, where the leading and sub-leading matrices satisfy the conditions in \eqref{CondicionesMatricesLideresFinales}, is a polynomial of degree $Np - r$.
\end{Proposition}
\begin{proof}
Expanding the determinant we get
	\begin{align*}
		\det  R(x) & =  \begin{vmatrix}
			\left[ R^1_{N-1} \right]x^{N-1}+O(x^{N-2}) & x^{N}I_{(p-r)} + O(x^{N-1}) \\
			\\
			x^{N-1}I_{r} + O(x^{N-2}) & \left[ R^3_{N-1} \right]x^{N-1}+O(x^{N-2})
		\end{vmatrix} \\
		& = \sum_{\sigma \in \mathcal{S}_p} \text{sgn}(\sigma) R_{1\sigma(1)}R_{2\sigma(2)} \cdots R_{p\sigma(p)} \\&
			= \text{sgn}(\Tilde{\sigma}) R_{1\Tilde{\sigma}(1)}R_{2\Tilde{\sigma}(2)} \cdots R_{(p-r)\Tilde{\sigma}(p-r)} \cdots R_{p\Tilde{\sigma}(p)} 
			+ \sum_{\sigma \neq \Tilde{\sigma}} \text{sgn}(\sigma) R_{1\sigma(1)}R_{2\sigma(2)} \cdots R_{p\sigma(p)},
	\end{align*}
	where the permutation $\Tilde{\sigma}(i)$ is such that $i \rightarrow r+i$ for $i \leq (p-r)$ and $i \rightarrow i-(p-r)$ for $i \geq (p-r)+1$. This term in the determinant expansion leads to a contribution of the form:
	\begin{multline*}
		\sgn(\Tilde{\sigma}) R_{1\Tilde{\sigma}(1)}R_{2\Tilde{\sigma}(2)} \cdots R_{(p-r)\Tilde{\sigma}(p-r)} \cdots R_{p\Tilde{\sigma}(p)} \\\begin{aligned}
			&= \sgn(\Tilde{\sigma}) R_{1,(r+1)}R_{2,(r+2)} \cdots R_{(p-r),p} R_{(p-r+1),1} \cdots R_{pr} =
			\sgn(\Tilde{\sigma}) (x^N)^{(p-r)}(x^{N-1})^{r} = x^{Np-r}.
		\end{aligned}
	\end{multline*} 	
	Any other permutation either results in zero or produces terms of lower degree.
\end{proof}
\begin{Proposition}
$R\left( \Lambda_{[p]}^\top \right)$ is banded lower triangular matrix that from the $Np-r$ subdiagonal is populated with zeros.
\end{Proposition}
\begin{proof}
We have
\begin{equation*}
	R\left( \Lambda_{[p]}^\top \right) = \left[\begin{NiceMatrix}
		R_0 & 0_p & 0_p & \Cdots \\
		R_1 & R_0 & 0_p & \Cdots \\
		\Vdots & & \Ddots & \\
		R_N & R_{N-1} & \Cdots & R_0 & \Cdots \\
		0_p & R_N & R_{N-1} & \Cdots & R_0 & \Cdots\\
		\Vdots & \Ddots[shorten-end=-10pt] &  \Ddots[shorten-end=-10pt] &  \Ddots[shorten-end=-25pt] &  & \Ddots
	\end{NiceMatrix} \right].
\end{equation*}
The block 
\begin{equation*}
	\begin{bmatrix}
		R_{N-1} & R_{N-2} \\
		R_N & R_{N-1}  
	\end{bmatrix} =\begin{bmatrix}
		\left[ R^1_{N-1} \right]_{(p-r)\times r} & & \left[ R^2_{N-1} \right]_{(p-r)\times (p-r)} & & &  \\ 
		& & & & R_{N-2} & \\
		I_{r\times r} & & \left[ R^4_{N-1} \right]_{r \times (p-r)} &  &  & \\ \\
		0_{(p-r)\times r} & & I_{(p-r)} & \left[ R^1_{N-1} \right]_{(p-r)\times r} & & \left[ R^2_{N-1} \right]_{(p-r)\times (p-r)} \\ \\
		0_{r\times r} & & 0_{r \times (p-r)} & I_{r\times r} & & \left[ R^4_{N-1} \right]_{r \times (p-r)}
	\end{bmatrix}
\end{equation*} has $p-r$ subdiagonals. Up to the matrix $R_{N-1}$, there will be $M-1$ matrices, which sum up to a total of $Np-r$ subdiagonals.
\end{proof}

\begin{Corollary} \label{ProyectorK-2}
    If the leading matrix of a given polynomial of degree $N$ is of the form: 
    \begin{equation*}
        P_N = \begin{bNiceMatrix}
            0_{(p-r) \times r} & I_{(p-r)} \\[2pt]
            0_r & 0_{r \times (p-r)}
        \end{bNiceMatrix},
    \end{equation*}
    then the projection property shown in Proposition \ref{ProyectorK}  holds for $n \geq Np + p -1 - r$.
\begin{proof}
    For $0 \leq i \leq p-1$, we have: 
    \begin{align*}
        \sum_{a=1}^p \sum_{b=1}^q \int_\Delta B_{Np+p-1-i}^{(b)}(t) \d \mu_{b,a}(t) \left( P(t) \right)_{a,\Tilde{a}} = \sum_{a=1}^p \left[ \sum_{b=1}^q \int_\Delta B_{Np+p-1-i}^{(b)}(t) \d \mu_{b,a}(t) t^N \right] \left( P_N \right)_{a,\Tilde{a}}.
    \end{align*}
    It is known that $\left( P_N \right)_{a,\Tilde{a}} = 0$ for $\Tilde{a}, a = \{ p-r+1, \dots, p \}$, 
    \begin{align*}
        \sum_{a=1}^p \left[ \sum_{b=1}^q \int_\Delta B_{Np+p-1-i}^{(b)}(t) \d \mu_{b,a}(t) t^N \right] \left( P_N \right)_{a,\Tilde{a}} = \sum_{a=1}^{p-r} \left[ \sum_{b=1}^q \int_\Delta B_{Np+p-1-i}^{(b)}(t) \d \mu_{b,a}(t) t^N \right] \left( P_N \right)_{a,\Tilde{a}},
    \end{align*}
    To eliminate the last term by using the orthogonality relations, we have to choose $i$ such that: 
    \begin{equation*}
        \ceil[\Big]{\frac{Np-1+i-a+2}{p}} - 1 \geq \ceil[\Big]{\frac{Np-1+i-p+r+2}{p}} - 1 \geq N. 
    \end{equation*}
    Therefore, $i \leq r$, and the result is independent of $n$ for $n \geq Np + p - 1 - r$.
\end{proof}
\end{Corollary}

\section{Geronimus Perturbation}

In this section, we analyze Geronimus perturbations, starting with the case of simple eigenvalues and extending to eigenvalues of arbitrary multiplicity. We also explore left multiplication within this framework and present a worked example, focusing on non-trivial perturbations of Jacobi–Piñeiro polynomials with three weights.

    \subsection{Simple Eigenvalues}
    
To illustrate our main results in detail, we will examine a specific case of matrix polynomial perturbation. Let us assume that all roots of the determinantal polynomial of $R(x)$ are simple, i.e., $M = Np - r$. A Geronimus perturbation of the measure is given by:
\begin{equation*}
    \d \check{\mu}(x)R(x) = \d \mu(x),
\end{equation*}
subject to the condition $\sigma(R) \cap \Delta = \emptyset$. The perturbed measure can be expressed in terms of $\mu$ and $R(x)$ as follows:
\begin{equation}\label{eq:perturbed measure}
    \check{\mu}(x) = \mu(x)R^{-1}(x) + \sum_{i=1}^M \boldsymbol{\xi}_i(x) \boldsymbol{v_{i}^L} \delta(x-x_i),
\end{equation}
where $\delta$ denotes Dirac's delta distribution, and $\boldsymbol{v}_{i}^L$ are the left eigenvectors of $R(x)$ corresponding to simple eigenvalues. Here, $\boldsymbol{\xi}_i$ is an arbitrary column vector function of size $q$. 
\begin{Proposition} \label{ConnectM}
    The moment matrices satisfy the following relation:
    \begin{equation*}
        \check{\mathscr{M}}R(\Lambda_{[p]}^\top) = \mathscr{M}.
    \end{equation*}
\end{Proposition}
\begin{proof}
    Recalling that $\Lambda_{[p]}X_{[p]}(x) = x X_{[p]}(x)$, we have:
    \begin{align*}
        \check{\mathscr{M}}R(\Lambda_{[p]}^\top) &= \int_\Delta X_{[q]}(x) \d \check{\mu}(x) X^\top_{[p]}(x) \sum_{i=0}^N R_i(\Lambda_{[p]}^\top)^i \\
        &= \int_\Delta X_{[q]}(x) \d \check{\mu}(x) \sum_{i=0}^N R_i X^\top_{[p]}(x)(\Lambda_{[p]}^\top)^i \\
        &= \int_\Delta X_{[q]}(x) \d \check{\mu}(x) \left( \sum_{i=0}^N R_i x^i \right) X^\top_{[p]}(x) = \mathscr{M},
    \end{align*}
    where $\C^{p \times \infty}[x]$ is considered a $\C^{p \times p}[x]$ bimodule.
\end{proof}
Assume the moment matrices have $LU$ factorizations:
\[\begin{aligned}
    \check{\mathscr{M}} &= \check{S}^{-1}\check{H} \check{\Bar{S}}^{-\top}, & \mathscr{M} &= S^{-1}H \Bar{S}^{-\top}.
\end{aligned}\]
That is, both lead to the existence of corresponding orthogonalities. 

\begin{Proposition} \label{MatrixConnect}
    There exists a matrix, referred to as the connection matrix, such that:
    \begin{align*}
        \Omega \coloneq \check{S}S^{-1} = \check{H}\check{\Bar{S}}^{\top} R(\Lambda_{[p]}^\top) \Bar{S}^{-\top} H^{-1}.
    \end{align*}
    The connection matrix is lower unitriangular, with at most $M$ nonzero subdiagonals. Explicitly: 
    \begin{align*}
		\Omega &= \begin{bNiceMatrix}
			1 & 0 & \Cdots[shorten-end=7pt] \\
			\Omega_{1,0} & 1 & 0 & \Cdots[shorten-end=7pt]  \\
			\Omega_{2,0} & \Omega_{2,1} & 1 & 0 & \Cdots[shorten-end=7pt] \\
			\Vdots & \Vdots & & \Ddots & \Ddots[shorten-end=8pt]  \\
			\Omega_{pN-r,0} & \Omega_{pN-r,1} & \Cdots & \Omega_{pN-r,pN-r-1} & 1 &  &  \\
			0 & \Omega_{pN-r+1,1} &  &  & \Omega_{pN-r+1,pN-r} & \Ddots[shorten-end=5pt] \\
			\Vdots[shorten-end=1pt] & \Ddots[shorten-end=-25pt] & \Ddots[shorten-end=-55pt]  & \Ddots[shorten-end=-55pt]  &  & \Ddots[shorten-end=-2pt]  \\
		\end{bNiceMatrix} \\&=  \begin{bNiceMatrix}
			\left[\Omega_{0,0}\right]_p& 0_p& \Cdots[shorten-end=7pt]  \\
			\left[\Omega_{1,0}\right]_p &\left[\Omega_{1,1}\right]_p & 0_p & \Cdots[shorten-end=7pt] \\
		\Vdots & \Vdots & \Ddots & \Ddots[shorten-end=5pt]\\[3pt]
				\left[\Omega_{N,0}\right]_p & \left[\Omega_{N,1}\right]_p & \Cdots &\left[\Omega_{N,N}\right]_p &  \\[3pt]
		0_p & \left[\Omega_{N+1,1}\right]_p & \Cdots & \left[\Omega_{N+1,N}\right]_p  & \Ddots[shorten-end=5pt]  \\
		\Vdots[shorten-end=1pt] & \Ddots[shorten-end=30pt] & \Ddots[shorten-end=-20pt]  & &  \Ddots[shorten-end=5pt]  \\
		\end{bNiceMatrix}.
	\end{align*}
Where $[\Omega_{n,m}]_p$ are square matrices with \( p \) rows and \( p \) columns.
\end{Proposition}
\begin{proof}
    The equality follows from Proposition \ref{ConnectM}. The structure of the connection matrix follows from the properties of the $LU$ matrix factors and the fact that $R(\Lambda_{[p]}^\top)$ has $M$ subdiagonals.
\end{proof}
The connection matrix is introduced to establish explicit relations or connection formulas between the original and perturbed families of orthogonal polynomials.
\begin{Proposition} \label{ConnectAB}
    The following relations hold:
\[    \begin{aligned}
        \check{A}(x) \Omega &= R(x) A(x), & \Omega B(x) &= \check{B}(x).
    \end{aligned}\]
\end{Proposition}
\begin{proof}
    These follow from the definitions of $\check{A}(x), \: A(x), \: \check{B}(x), \: B(x)$, and Proposition \ref{MatrixConnect}. For instance,
    \begin{equation*}
        \check{B}(x) = \check{S}X_{[q]}(x) = \check{S}S^{-1}SX_{[q]}(x) = \Omega B(x).
    \end{equation*}
\end{proof}
These connection formulas, however, are not particularly useful for our purpose. We aim to express the entries of the connection matrix in terms of the original orthogonal polynomials. One possible approach involves applying the different eigenvectors of $R(x)$ to the $A(x)$ connection formula and evaluating the resulting expressions at the corresponding eigenvalues. However, this leads to the equation $\boldsymbol{v}_{i}^L R(x_i)A(x_i) = 0 = \boldsymbol{v}_{i}^L\check{A}(x_i) \Omega$, which does not allow us to solve the system of linear equations.

Cauchy transforms will be highly useful in this context, facilitating the derivation of Christoffel formulas.

\begin{Proposition} \label{ConnectCD}
    In terms of the Cauchy transforms of the original and perturbed orthogonal polynomials, the connection formulas are given by:
    \begin{align}
        \label{EqConnectC} \check{C}(x) \Omega &= C(x), \\
        \label{EqConnectD} \check{D}(x) R(x) &= \Omega D(x) + \int_\Delta \check{B}(y) \d \check{\mu}(y) \frac{R(x) - R(y)}{x-y}.
    \end{align}
\end{Proposition}
\begin{proof}
    We will prove the second relation, as the first can be shown similarly. We have:
    \begin{align*}
        \Omega D(x) &= \int_\Delta \Omega B(y) \frac{\d \mu(y)}{x-y} = \int_\Delta \check{B}(y) \frac{\d \mu(y)}{x-y} \\
        &=  \int_\Delta \check{B}(y) \frac{\d \check{\mu}(y)R(y)}{x-y} \\
        &= \check{D}(x)R(x) - \int_\Delta \check{B}(y) \frac{\d \check{\mu}(y)R(x)}{x-y} + \int_\Delta \check{B}(y) \frac{\d \check{\mu}(y)R(y)}{x-y}.
    \end{align*}
\end{proof}
\begin{Lemma}\label{R-BivaMP}
  The matrix function   $\frac{R(x) - R(y)}{x-y}$ is a bivariate matrix polynomial of degree $N-1$ in both variables. 
\end{Lemma}
\begin{proof}
    Using the identity:
    \begin{equation*}
        x^n - y^n = (x-y)\left( \sum_{i=1}^n x^{n-i}y^{i-1} \right),
    \end{equation*}
   it is straightforward to verify that:
   \begin{equation}\label{eq:Rxy}
       \frac{R(x) - R(y)}{x-y} = \sum_{i=1}^N R_i \sum_{j=1}^i \left( x^{i-j}y^{j-1} \right).
   \end{equation}
   Hence, it is a polynomial in both $x$ and $y$. 
\end{proof}
Let's examine the commutator, $[\Omega,\Pi_n]$, where $\Pi_n$ is the diagonal matrix with all entries zero, except for the first $n+1$ entries, which are equal to unity. We have:
\begin{align} \label{Commutator}
    [\Omega,\Pi_n] &= \Omega\Pi_n - \Omega^{[n]} \\ &= \left[\begin{NiceMatrix}[columns-width=1.2cm]
	0 & \Cdots & 0 & \Cdots \\
	\Vdots & & \Vdots \\
	0 & \Cdots & \Omega_{n+1,n-M+1} & \Cdots & & \Omega_{n+1,n} & 0 & \Cdots \\
	\Vdots & & 0 & \Ddots & & \Vdots & \Vdots & \\
	& & \Vdots & & & & \\
	& & & & \Ddots[shorten-end=20pt] & \Omega_{n+M,n} &&\\\\
	0 & \Cdots & 0 & \Cdots & & 0 &&\\
	\Vdots[shorten-end=2pt] & & \Vdots[shorten-end=2pt] & & & \Vdots[shorten-end=2pt]
    \end{NiceMatrix}\right],
\end{align}
where $\Omega_{n+1,n-M+1}$ occupies the $(n,n-M)$ entry (counting $n=0$ as the first entry) of the matrix. For $n < M$, we assume that the commutator starts at the first column, and any negative indices should be disregarded.
\begin{Definition}\label{D and W bb}
  For $i\in\{1,\dots,M\}$ and $n\in\N_0$  we introduce the following notation
\[    \begin{aligned}
        \mathbb{D}_n^{(i)} & \coloneq \sum_{a=1}^p D_n^{(a)}(x_i)v_{i;a}^R, & \mathbb{W}_n^{(i)} & \coloneq \sum_{b=1}^q B_n^{(b)}(x_i)\xi_{i;b}(x_i)\boldsymbol{v}_{i}^L R'(x_i) \boldsymbol{v}_{i}^R. 
    \end{aligned}  \]
\end{Definition}
Let us now solve the linear system for the components of $\Omega$ using Proposition \ref{ConnectCD}. 
\begin{Proposition} \label{OmegaComponents}
    For $n \geq M$, the entries of $\Omega$ are subject to the inhomogeneous linear system
    \begin{equation} \label{EqOmegaComponents-n>M}
        \begin{bNiceMatrix}
            \Omega_{n,n-M} & \Cdots & \Omega_{n,n-1} 
        \end{bNiceMatrix}  \begin{bNiceMatrix}
        \mathbb{D}_{n-M}^{(1)}- \mathbb{W}_{n-M}^{(1)} & \Cdots & \mathbb{D}_{n-M}^{(M)} - \mathbb{W}_{n-M}^{(M)}\\
        \Vdots & & \Vdots\\
        \mathbb{D}_{n-1}^{(1)} - \mathbb{W}_{n-1}^{(1)}& \Cdots & \mathbb{D}_{n-1}^{(M)} - \mathbb{W}_{n-1}^{(M)}
        \end{bNiceMatrix}= - \begin{bNiceMatrix}
            \mathbb{D}_n^{(1)}- \mathbb{W}_n^{(1)}& \Cdots & \mathbb{D}_n^{(M)} - \mathbb{W}_n^{(M)}
        \end{bNiceMatrix}. 
    \end{equation}
\end{Proposition}
\begin{proof}
    We cannot directly evaluate Equation \eqref{EqConnectD} at $x_i$, so we examine the following limit:
    \begin{align*}
        \lim_{x \rightarrow x_i} \check{D}(x) R(x) \boldsymbol{v}_{i}^R &= \lim_{x \rightarrow x_i} \left( \int_\Delta \check{B}(y) \frac{\d \mu(y)}{x-y}R^{-1}(y)R(x) + \sum_{j=1}^M \check{B}(x_j) \boldsymbol{\xi}(x_j)_j\boldsymbol{v}_{j}^L\frac{R(x)}{x-x_j}  \right) \boldsymbol{v}_{i}^R \\
        &= \check{B}(x_i) \boldsymbol{\xi}(x_i)_i\boldsymbol{v}_{i}^L \lim_{x \rightarrow x_i} \frac{R(x)-R(x_i)}{x-x_i} \boldsymbol{v}_{i}^R,
    \end{align*}
    where we have used the fact that $R(x_i)\boldsymbol{v}_{i}^R = 0$. The term that remains non-zero can be rewritten as
    \begin{equation*}
        \lim_{x \rightarrow x_i} \check{D}(x) R(x) \boldsymbol{v}_{i}^R = \check{B}(x_i) \boldsymbol{\xi}(x_i)_i\boldsymbol{v}_{i}^L R'(x_i) \boldsymbol{v}_{i}^R. 
    \end{equation*} 
    
    Let us now focus on the case $n \geq M$. By using the orthogonality relations for $B(x)$ and taking into account Proposition \ref{R-BivaMP} (see proof of Corollary \ref{ProyectorK-2}),
    \begin{equation*}
\begin{aligned}
	        \sum_{a=1}^p\sum_{b=1}^q \int_\Delta \check{B}_n^{(b)}(y) \d \check{\mu}_{b,a}(y) \frac{\left( R(x)-R(y) \right)_{a,\Tilde{a}}}{x-y} &= 0, & \Tilde{a}&\in \{1,\dots, p\}.
\end{aligned}
    \end{equation*}
    Taking the limit in Equation \eqref{EqConnectD} when $x \rightarrow x_i$, then multiplying by the corresponding right eigenvector, and finally considering the $n$-th component, we find
    \begin{equation*}
        \sum_{b=1}^q\check{B}_n^{(b)}(x_i) \xi_{i;b}(x_i)\boldsymbol{v}_{i}^L R'(x_i) \boldsymbol{v}_{i}^R = \left( \Omega D(x_i) \boldsymbol{v}_{i}^R \right)^{[n]}.  
    \end{equation*}
    Recalling Proposition \ref{ConnectAB}, we obtain
    \begin{multline*}
         \sum_{b=1}^q\left( B_n^{(b)}(x_i) + \Omega_{n,n-1}B_{n-1}^{(b)}(x_i) + \cdots + \Omega_{n,n-M}B_{n-M}^{(b)}(x_i) \right)\xi_{i;b}\boldsymbol{v}_{i}^L R'(x_i) \boldsymbol{v}_{i}^R   \\
         = \left( D_n^{(a)}(x_i) + \Omega_{n,n-1}D_{n-1}^{(a)}(x_i) + \cdots + \Omega_{n,n-M}D_{n-M}^{(a)}(x_i) \right) v_{i;a}^R,
    \end{multline*}
    and using the notation introduced, we can write
    \begin{equation*}
        \mathbb{W}_n^{(i)}+ \Omega_{n,n-1}\mathbb{W}_{n-1}^{(i)}+ \cdots + \Omega_{n,n-M}\mathbb{W}_{n-M}^{(i)}= \mathbb{D}_n^{(i)}+ \Omega_{n,n-1}\mathbb{D}_{n-1}^{(i)}+ \cdots + \Omega_{n,n-M}\mathbb{D}_{n-M}^{(i)}.
    \end{equation*}
    Rearranging the equation yields
    \begin{equation*}
        - \left( \mathbb{D}_n^{(i)}- \mathbb{W}_n^{(i)}\right) = \begin{bNiceMatrix}
            \Omega_{n,n-M} & \Cdots & \Omega_{n,n-1} 
        \end{bNiceMatrix}\begin{bNiceMatrix}
            \mathbb{D}_{n-M}^{(i)}- \mathbb{W}_{n-M}^{(i)}\\
            \Vdots \\
            \mathbb{D}_{n-1}^{(i)}- \mathbb{W}_{n-1}^{(i)}
        \end{bNiceMatrix}. 
    \end{equation*}
    Considering the $M$ distinct roots, it is straightforward to arrive at the stated equation.
\end{proof}
\begin{Definition}
    Let us introduce the following determinantal expressions:
    \begin{equation*}
\begin{aligned}
	        \tau_n^{(i)} &\coloneq \begin{vNiceMatrix}
            \mathbb{D}_{n-M}^{(1)} - \mathbb{W}_{n-M}^{(1)}& \Cdots & \mathbb{D}_{n-M}^{(M)}- \mathbb{W}_{n-M}^{(M)} \\
            \Vdots & & \Vdots \\
            \mathbb{D}_{n-i-1}^{(1)}- \mathbb{W}_{n-i-1}^{(1)}& \Cdots & \mathbb{D}_{n-i-1}^{(M)}- \mathbb{W}_{n-i-1}^{(M)}\\[5pt]
            \mathbb{D}_{n-i+1}^{(1)}- \mathbb{W}_{n-i+1}^{(1)}& \Cdots & \mathbb{D}_{n-i+1}^{(M)}- \mathbb{W}_{n-i+1}^{(M)}\\
            \Vdots & & \Vdots \\
            \mathbb{D}_{n}^{(1)}- \mathbb{W}_{n}^{(1)}& \Cdots & \mathbb{D}_{n}^{(M)}- \mathbb{W}_{n}^{(M)}
        \end{vNiceMatrix}, & i& \in \{0, \cdots, M \}, & n & \in \{M,M+1,M+2,\dots\}.
\end{aligned}
    \end{equation*}
    For \( i = M \), we will also use the notation $\tau_{n} \coloneq \tau_{n+1}^{(M)}$, $ n  \in\{M-1,M,M+1,\dots\}$.
\end{Definition}


\begin{Proposition} \label{ConnectKernelmix}
    The connection formulas for the mixed \( CD \) kernels, for \( n \geq M \), are given by:
    \begin{multline} \label{EqConnectKernelmix-n>M}
        \check{K}^{[n-1]}_{D}(x,y)R(y) = \frac{R(x) - R(y)}{x - y} + R(x)K^{[n-1]}_D (x,y) \\[3pt]
        - \begin{bNiceMatrix}
            \check{A}_n^{(1)}(x)  & \Cdots & \check{A}_{n+M-1}^{(1)}(x) \\
            \Vdots & & \Vdots \\
            \check{A}_n^{(p)}(x) & \Cdots & \check{A}_{n+M-1}^{(p)}(x)
        \end{bNiceMatrix}
        \begin{bNiceMatrix}
            \Omega_{n,n-M} & \Cdots & & \Omega_{n,n-1} \\
            0 & & & \\
            \Vdots & \Ddots & \Ddots[shorten-end=-8pt] & \Vdots \\
            0 & \Cdots & 0 & \Omega_{n+M-1,n-1}
        \end{bNiceMatrix}
        \begin{bNiceMatrix}
            D^{(1)}_{n-M}(y) & \Cdots & D^{(q)}_{n-M}(y) \\
            \Vdots & & \Vdots \\
            D^{(1)}_{n-1}(y) & \Cdots & D^{(q)}_{n-1}(y)
        \end{bNiceMatrix}. 
    \end{multline}
\end{Proposition}

\begin{proof}
    For \( n \geq M \), using Equation \eqref{EqConnectD}, we find:
    \begin{equation*}
        \check{K}^{[n-1]}_{D}(x,y)R(y) = \check{A}(x) \Pi_{n-1} \check{D}(y)R(y) = \check{A}(x) \Pi_{n-1}  \Omega D(y) + \check{A}(x)\Pi_{n-1}  \int_{\Delta} \check{B}(t) \, \d \mu(t) \frac{R(t) - R(y)}{t - y}.
    \end{equation*}
    Using Proposition \ref{ProyectorK} and \ref{R-BivaMP} , the last term equals \( \frac{R(x) - R(y)}{x - y} \). On the other hand,
    \begin{equation*}
        \check{A}(x) \Pi_{n-1}  \Omega D(y) = \check{A}(x) \Omega \Pi_{n-1}  D(y) - \check{A}(x) [\Omega, \Pi_{n-1} ] D(y).
    \end{equation*}
    From Proposition \ref{ConnectAB} and Equation \eqref{Commutator}, we obtain
    \begin{multline*}
        \check{A}(x) \Pi_{n-1}  \Omega D(y) = R(x)K^{[n-1]}_D (x,y) 
  - \begin{bNiceMatrix}
            \check{A}_n^{(1)}(x)  & \Cdots & \check{A}_{n+M-1}^{(1)}(x) \\
            \Vdots & & \Vdots \\
            \check{A}_n^{(p)}(x) & \Cdots & \check{A}_{n+M-1}^{(p)}(x)
        \end{bNiceMatrix}
        \begin{bNiceMatrix}
            \Omega_{n,n-M} & \Cdots & & \Omega_{n,n-1} \\
            0 & & & \\
            \Vdots & \Ddots & \Ddots[shorten-end=-8pt] & \Vdots \\
            0 & \Cdots & 0 & \Omega_{n+M-1,n-1}
        \end{bNiceMatrix} \\
 \times \begin{bNiceMatrix}
            D^{(1)}_{n-M}(y) & \Cdots & D^{(q)}_{n-M}(y) \\
            \Vdots & & \Vdots \\
            D^{(1)}_{n-1}(y) & \Cdots & D^{(q)}_{n-1}(y)
        \end{bNiceMatrix}.
    \end{multline*}
    Thus, the desired relation follows immediately.
\end{proof}
\begin{Definition} \label{K bb}
   Let us introduce the following vector polynomials: 
    \begin{equation*}
        \mathbb{K}^{[n-1],(i)}(x) \coloneq K^{[n-1]}_D (x,x_i)\boldsymbol{v}_{i}^R - K^{[n-1]}(x,x_i)\boldsymbol{\xi}_i \boldsymbol{v}_{i}^L R'(x_i) \boldsymbol{v}_{i}^R.
    \end{equation*}
    \begin{Proposition} \label{Futuro Tau}
    	    The following holds
    	    	\begin{multline} \label{EqSistLinealTau0}  
    	    		R(x)  \begin{bNiceMatrix}
    	    			\mathbb{K}^{[n-1],(1)}(x) + \frac{\boldsymbol{v}_{1}^R}{x-x_1} & \Cdots & \mathbb{K}^{[n-1],(M)}(x) + \frac{\boldsymbol{v}_{M}^R}{x-x_M}
    	    		\end{bNiceMatrix} \\ = \begin{bNiceMatrix}
    	    			\check{A}_n^{(1)}(x)  & \Cdots & \check{A}_{n+M-1}^{(1)}(x) \\
    	    			\Vdots & & \Vdots \\
    	    			\check{A}_n^{(p)}(x) & \Cdots & \check{A}_{n+M-1}^{(p)}(x)
    	    		\end{bNiceMatrix} \begin{bNiceMatrix}
    	    			\Omega_{n,n-M} & \Cdots & & \Omega_{n,n-1} \\
    	    			0 & & &   \\
    	    			\Vdots & \Ddots & \Ddots[shorten-end=-8pt] & \Vdots \\
    	    			0 & \Cdots & 0 & \Omega_{n+M-1,n-1}
    	    		\end{bNiceMatrix} \\\times\begin{bNiceMatrix}
    	    		\mathbb{D}_{n-M}^{(1)}- \mathbb{W}_{n-M}^{(1)}& \Cdots & \mathbb{D}_{n-M}^{(M)}- \mathbb{W}_{n-M}^{(M)}\\
    	    		\Vdots & & \Vdots \\
    	    		\mathbb{D}_{n-1}^{(1)}- \mathbb{W}_{n-1}^{(1)}& \Cdots & \mathbb{D}_{n-1}^{(M)}- \mathbb{W}_{n-1}^{(M)}
    	    		\end{bNiceMatrix}.
    	    \end{multline}
    \end{Proposition}
    \begin{proof}
    	 First, consider:
    	\begin{align*}
    			\lim_{y \to x_i} \check{K}^{[n-1]}_D (x,y)R(y)\boldsymbol{v}_{i}^R &= \lim_{y \to x_i} \check{A}(x) \Pi_{n-1}  \check{D}(y)R(y)\boldsymbol{v}_{i}^R \\
    			&= \begin{multlined}[t][.7\textwidth]
    				\check{A}(x) \Pi_{n-1}  \lim_{y \to x_i} \left( \Omega \int_\Delta B(t) \frac{\d \mu(t)}{y-t} R^{-1}(t)R(y)\boldsymbol{v}_{i}^R \right) \\
    			+ \check{A}(x) \Pi_{n-1}  \lim_{y \to x_i} \left( \sum_{j=1}^M \check{B}(x_j) \boldsymbol{\xi}_j(x_j)\boldsymbol{v}_{j}^L \frac{R(y)}{y-x_j}\boldsymbol{v}_{i}^R \right).
    			\end{multlined}
    	\end{align*}
    	Once again, the only nonzero term is the one corresponding to \( j = i \):
    	\begin{align*}
    		\lim_{y \to x_i} \check{K}^{[n-1]}_D (x,y)R(y)\boldsymbol{v}_{i}^R &= \check{A}(x) \Pi_{n-1}  \check{B}(x_i) \boldsymbol{\xi}_i(x_i)\boldsymbol{v}_{i}^L R'(x_i) \boldsymbol{v}_{i}^R \\
    		&= R(x)K^{[n-1]}(x,x_i)\boldsymbol{\xi}_i(x_i)\boldsymbol{v}_{i}^L R'(x_i) \boldsymbol{v}_{i}^R - \check{A}(x)[\Omega, \Pi_{n-1} ] B(x_i) \boldsymbol{\xi}_i(x_i)\boldsymbol{v}_{i}^L R'(x_i) \boldsymbol{v}_{i}^R.
    	\end{align*} 
    	
    	Taking the limit in Equation \eqref{EqConnectKernelmix-n>M} for \( y\to x_i \) and applying the corresponding Jordan chain:
    	\begin{multline*}
    			R(x)K^{[n-1]}(x,x_i)\boldsymbol{\xi}_i \boldsymbol{v}_{i}^L R'(x_i) \boldsymbol{v}_{i}^R -  \check{A}(x)[\Omega, \Pi_{n-1} ] B(x_i) \boldsymbol{\xi}_i(x_i)\boldsymbol{v}_{i}^L R'(x_i) \boldsymbol{v}_{i}^R  \\
    			= \frac{R(x)\boldsymbol{v}_{i}^R}{x-x_i} + R(x)K^{[n-1]}_D (x,x_i)\boldsymbol{v}_{i}^R - \check{A}(x) \left[ \Omega, \Pi_{n-1}  \right]D(x_i)\boldsymbol{v}_{i}^R.
    	\end{multline*}
    	Rearranging terms and introducing the notation mentioned above:
    	\begin{multline*}
    			R(x)\left( \mathbb{K}^{[n-1],(i)}(x) + \frac{\boldsymbol{v}_{i}^R}{x-x_i} \right) = \begin{bNiceMatrix}
    				\check{A}_n^{(1)}(x)  & \Cdots & \check{A}_{n+M-1}^{(1)}(x) \\
    				\Vdots & & \Vdots \\
    				\check{A}_n^{(p)}(x) & \Cdots & \check{A}_{n+M-1}^{(p)}(x)
    			\end{bNiceMatrix} \begin{bNiceMatrix}
    				\Omega_{n,n-M} & \Cdots & & \Omega_{n,n-1} \\
    				0 & & &   \\
    				\Vdots & \Ddots & \Ddots[shorten-end=-8pt] & \Vdots \\
    				0 & \Cdots & 0 & \Omega_{n+M-1,n-1}
    			\end{bNiceMatrix} \\
    			\times \begin{bNiceMatrix}
    				\mathbb{D}_{n-M}^{i)} - \mathbb{W}_{n-M}^{(i)} \\
    				\Vdots \\
    				\mathbb{D}_{n-1}^{(i)} - \mathbb{W}_{n-1}^{(i)}
    			\end{bNiceMatrix}. 
    	\end{multline*}

    \end{proof}
    
\end{Definition}

\begin{Theorem} \label{ExistsTheorem}
	If the perturbed orthogonality exists then $\tau_n\neq 0$, $ n  \in\{M-1,M,M+1,\dots\}$.
\end{Theorem}
\begin{proof}
	Let us assume that the orthogonality exists even though $\tau_{n-1} = 0$. Since $\tau_{n-1} = 0$, there exists a non-zero constant vector such that:
	\begin{equation*}
		\begin{bNiceMatrix}
			\mathbb{D}_{n-M}^{(1)}- \mathbb{W}_{n-M}^{(1)}& \Cdots & \mathbb{D}_{n-M}^{(M)}- \mathbb{W}_{n-M}^{(M)}\\
			\Vdots & & \Vdots \\
			\mathbb{D}_{n-1}^{(1)}- \mathbb{W}_{n-1}^{(1)}& \Cdots & \mathbb{D}_{n-1}^{(M)}- \mathbb{W}_{n-1}^{(M)}
		\end{bNiceMatrix} \boldsymbol{c} = 0. 
	\end{equation*}
	Considering \eqref{EqSistLinealTau0}, the following relation must also hold:
	\begin{equation*}
		R(x)\begin{bNiceMatrix}
			\mathbb{K}^{[n-1],(1)}(x) + \frac{\boldsymbol{v}_{1}^R}{x-x_1} & \Cdots & \mathbb{K}^{[n-1],(M)}(x) + \frac{\boldsymbol{v}_{M}^R}{x-x_M}
		\end{bNiceMatrix} \boldsymbol{c} = 0. 
	\end{equation*} 
For $x \notin \sigma\left( R \right)$, the last relation now reads:
	\begin{equation*}
		\begin{bNiceMatrix}
			\mathbb{K}^{[n-1],(1)}(x) + \frac{\boldsymbol{v}_{1}^R}{x-x_1} & \Cdots & \mathbb{K}^{[n-1],(M)}(x) + \frac{\boldsymbol{v}_{M}^R}{x-x_M}
		\end{bNiceMatrix} \boldsymbol{c} = 0. 
	\end{equation*}
	Let us integrate the last relation in the complex plane using a counterclockwise  contour $C_i$ surrounding once  $x_i$ and not any other eigenvalue: 
	\begin{align*}
		\oint_{C_i} \begin{bNiceMatrix}
			\mathbb{K}^{[n-1],(1)}(z) + \frac{\boldsymbol{v}_{1}^R}{z-x_1} & \Cdots & \mathbb{K}^{[n-1],(M)}(z) + \frac{\boldsymbol{v}_{M}^R}{z-x_M}
		\end{bNiceMatrix} \boldsymbol{c} \: \d z = 2\pi \mathrm{i} \begin{bNiceMatrix}
			0 & \Cdots &0& \boldsymbol{v}_{i}^R & 0&\Cdots & 0 
		\end{bNiceMatrix} \boldsymbol{c} = 0. 
	\end{align*}
	Which can only be true if $\boldsymbol{v}_{i}^R\times c_i = 0$, or equivalently, $\boldsymbol{c} = 0$ (remember $\boldsymbol{v}_{i}^R$ is a vector). We arrive at a contradiction and therefore, the initial assumption must be false.
\end{proof}

\begin{Theorem} \label{ExplicitFormulas}
 Let's assume that the perturbed orthogonality exists.   For $n\geq M$, the following relations hold:
\begin{equation}
	    \begin{aligned}
        \label{EqOmegaExplicit} \Omega_{n,n-M} &= (-1)^M \frac{\tau_n}{\tau_{n-1}}, &
        \Omega_{n,n-i} &= (-1)^i \frac{\tau_n^{(i)}}{\tau_{n-1}}.
    \end{aligned}
\end{equation}
Moreover, we have the following Christoffel type formulas
    \begin{align}
    \label{EqCheckA} \check{A}^{(a)}_n(x) & = \sum_{\Tilde{a}=1}^p \frac{R_{a,\Tilde{a}}(x)}{\tau_n} \begin{vNiceMatrix}
        \mathbb{D}_{n-M+1}^{(1)}- \mathbb{W}_{n-M+1}^{(1)}& \Cdots & \mathbb{D}_{n-M+1}^{(M)}- \mathbb{W}_{n-M+1}^{(M)}\\
        \Vdots & & \Vdots \\
        \mathbb{D}_{n-1}^{(1)}- \mathbb{W}_{n-1}^{(1)}& \Cdots & \mathbb{D}_{n-1}^{(M)}- \mathbb{W}_{n-1}^{(M)}\\
        \mathbb{K}_{\Tilde{a}}^{[n-M],(1)}(x) + \frac{v_{1;\Tilde{a}}^R}{x-x_1} & \Cdots & \mathbb{K}_{\Tilde{a}}^{[n-M],(M)}(x) + \frac{v_{M;\Tilde{a}}^R}{x-x_M}
    \end{vNiceMatrix}, \\
    \label{EqCheckB} \check{B}^{(b)}_{n}(x) & = \frac{1}{\tau_{n-1}}\begin{vNiceMatrix}
        \mathbb{D}_{n-M}^{(1)}- \mathbb{W}_{n-M}^{(1)}& \Cdots & \mathbb{D}_{n-M}^{(M)}- \mathbb{W}_{n-M}^{(M)}& B_{n-M}^{(b)}(x) \\
        \Vdots & & \Vdots & \Vdots \\
        \mathbb{D}_{n-1}^{(1)}- \mathbb{W}_{n-1}^{(1)}& \Cdots & \mathbb{D}_{n-1}^{(M)}- \mathbb{W}_{n-1}^{(M)}& B_{n-1}^{(b)}(x) \\[4pt]
        \mathbb{D}_{n}^{(1)}- \mathbb{W}_{n}^{(1)} & \Cdots & \mathbb{D}_{n}^{(M)}- \mathbb{W}_{n}^{(M)}& B_n^{(b)}(x)
    \end{vNiceMatrix}.
    \end{align}
\end{Theorem}
\begin{proof}
    We will begin by proving Equation \eqref{EqOmegaExplicit}. Multiply by the vector \( \left[\begin{NiceMatrix} 1 & 0 & \Cdots \end{NiceMatrix}\right]^\top \) to obtain:
    \begin{equation*}
        \Omega_{n,n-M} = - \begin{bNiceMatrix}
            \mathbb{D}_n^{(1)}- \mathbb{W}_n^{(1)}& \Cdots & \mathbb{D}_n^{(M)}- \mathbb{W}_n(x_M)
        \end{bNiceMatrix}  \begin{bNiceMatrix}
            \mathbb{D}_{n-M}^{(1)}- \mathbb{W}_{n-M}^{(1)}& \Cdots & \mathbb{D}_{n-M}^{(M)}- \mathbb{W}_{n-M}^{(M)}\\
            \Vdots & & \Vdots \\
            \mathbb{D}_{n-1}^{(1)}- \mathbb{W}_{n-1}^{(1)}& \Cdots & \mathbb{D}_{n-1}^{(M)}- \mathbb{W}_{n-1}^{(M)}
        \end{bNiceMatrix}^{-1} \begin{bNiceMatrix}
            1 \\
            0 \\
            \Vdots[shorten-end=2pt]
        \end{bNiceMatrix}.
    \end{equation*}
    The last expression can be rewritten as:
    \begin{equation*}
        \Omega_{n,n-M} = \frac{\begin{vNiceMatrix}
            \mathbb{D}_{n-M}^{(1)}- \mathbb{W}_{n-M}^{(1)}& \Cdots & \mathbb{D}_{n-M}^{(M)}- \mathbb{W}_{n-M}^{(M)}& 1 \\
            \Vdots & & \Vdots & 0 \\
            \mathbb{D}_{n-1}^{(1)}- \mathbb{W}_{n-1}^{(1)}& \Cdots & \mathbb{D}_{n-1}^{(M)}- \mathbb{W}_{n-1}^{(M)}& \Vdots \\[4pt]
            \mathbb{D}_n^{(1)}- \mathbb{W}_n^{(1)}& \Cdots & \mathbb{D}_n^{(M)}- \mathbb{W}_n^{(M)}& 0
        \end{vNiceMatrix}}{\tau_{n-1}} = (-1)^M \frac{\tau_n}{\tau_{n-1}}.
    \end{equation*}
    The second relation in \eqref{EqOmegaExplicit} can be proven similarly. We now proceed to prove Equation \eqref{EqCheckA}. 
   From Equation \eqref{EqSistLinealTau0} we obtain
       	\begin{multline*}
   		R(x)  \begin{bNiceMatrix}
   			\mathbb{K}^{[n-1],(1)}(x) + \frac{\boldsymbol{v}_{1}^R}{x-x_1} & \Cdots & \mathbb{K}^{[n-1],(M)}(x) + \frac{\boldsymbol{v}_{M}^R}{x-x_M}
   		\end{bNiceMatrix} 
   	\begin{bNiceMatrix}
   			\mathbb{D}_{n-M}^{(1)}- \mathbb{W}_{n-M}^{(1)}& \Cdots & \mathbb{D}_{n-M}^{(M)}- \mathbb{W}_{n-M}^{(M)}\\
   			\Vdots & & \Vdots \\
   			\mathbb{D}_{n-1}^{(1)}- \mathbb{W}_{n-1}^{(1)}& \Cdots & \mathbb{D}_{n-1}^{(M)}- \mathbb{W}_{n-1}^{(M)}
   		\end{bNiceMatrix}^{-1} \\= \begin{bNiceMatrix}
   			\check{A}_n^{(1)}(x)  & \Cdots & \check{A}_{n+M-1}^{(1)}(x) \\
   			\Vdots & & \Vdots[shorten-end=-5pt] \\
   			\check{A}_n^{(p)}(x) & \Cdots & \check{A}_{n+M-1}^{(p)}(x)
   		\end{bNiceMatrix} \left[ \begin{NiceMatrix}
   			\Omega_{n,n-M} & \Cdots & & \Omega_{n,n-1} \\
   			0 & & &   \\
   			\Vdots & \Ddots & \Ddots[shorten-end=-8pt] & \Vdots \\
   			0 & \Cdots & 0 & \Omega_{n+M-1,n-1}
   		\end{NiceMatrix}\right].
   \end{multline*}
   
    Right-multiplying both terms by a matrix with a non-zero entry only in the first position yields:
    \begin{multline*}
         R(x) \begin{bNiceMatrix}
            \mathbb{K}^{[n-1],(1)}(x) + \frac{\boldsymbol{v}_{1}^R}{x-x_1} & \Cdots & \mathbb{K}^{[n-1],(M)}(x) + \frac{\boldsymbol{v}_{M}^R}{x-x_M}
        \end{bNiceMatrix} \\ 
        \times \begin{bNiceMatrix}
            \mathbb{D}_{n-M}^{(1)}- \mathbb{W}_{n-M}^{(1)}& \Cdots & \mathbb{D}_{n-M}^{(M)}- \mathbb{W}_{n-M}^{(M)}\\
            \Vdots & & \Vdots \\
            \mathbb{D}_{n-1}^{(1)}- \mathbb{W}_{n-1}^{(1)}& \Cdots & \mathbb{D}_{n-1}^{(M)}- \mathbb{W}_{n-1}^{(M)}
        \end{bNiceMatrix}^{-1} \left[\begin{NiceMatrix}
            1 & 0 & \Cdots[shorten-end=2pt] \\
            0 & 0& \Cdots[shorten-end=2pt] \\
            \Vdots[shorten-end=2pt]&      \Vdots[shorten-end=2pt]&
        \end{NiceMatrix} \right]= \Omega_{n,n-M} \begin{bNiceMatrix}
            \check{A}_n^{(1)}(x) & 0 & \Cdots[shorten-end=3pt] \\ 
            \Vdots & \Vdots \\ 
            \check{A}_n^{(p)}(x) & 0 & \Cdots[shorten-end=3pt] 
        \end{bNiceMatrix}.
    \end{multline*}
    Focusing only on the term \( \check{A}_n^{(a)}(x) \):
    \begin{multline*}
         \sum_{\Tilde{a}=1}^p R_{a,\Tilde{a}}(x) \begin{bNiceMatrix}
            \mathbb{K}_{\Tilde{a}}^{[n-1],(1)}(x) + \frac{v_{1;\Tilde{a}}^R}{x-x_1} & \Cdots & \mathbb{K}_{\Tilde{a}}^{[n-1],(M)}(x) + \frac{v_{M;\Tilde{a}}^R}{x-x_M}
        \end{bNiceMatrix} \\ 
        \times \begin{bNiceMatrix}
            \mathbb{D}_{n-M}^{(1)}- \mathbb{W}_{n-M}^{(1)}& \Cdots & \mathbb{D}_{n-M}^{(M)}- \mathbb{W}_{n-M}^{(M)}\\
            \Vdots & & \Vdots \\
            \mathbb{D}_{n-1}^{(1)}- \mathbb{W}_{n-1}^{(1)}& \Cdots & \mathbb{D}_{n-1}^{(M)}- \mathbb{W}_{n-1}^{(M)}
        \end{bNiceMatrix}^{-1} \begin{bNiceMatrix}
            1 \\
            0 \\
            \Vdots[shorten-end=2pt]
        \end{bNiceMatrix} = \Omega_{n,n-M}\check{A}_n^{(a)}(x).
    \end{multline*}
    From this relation and Equation \eqref{EqOmegaExplicit}, one can easily obtain the following result,
    \begin{equation*}
        \check{A}^{(a)}_n(x) = \sum_{\Tilde{a}=1}^p \frac{R_{a,\Tilde{a}}(x)}{\tau_n} \begin{vNiceMatrix}
        \mathbb{D}_{n-M+1}^{(1)}- \mathbb{W}_{n-M+1}^{(1)}& \Cdots & \mathbb{D}_{n-M+1}^{(M)}- \mathbb{W}_{n-M+1}^{(M)}\\
        \Vdots & & \Vdots \\
        \mathbb{D}_{n-1}^{(1)}- \mathbb{W}_{n-1}^{(1)}& \Cdots & \mathbb{D}_{n-1}^{(M)}- \mathbb{W}_{n-1}^{(M)}\\
        \mathbb{K}_{\Tilde{a}}^{[n-1],(1)}(x) + \frac{v_{1;\Tilde{a}}^R}{x-x_1} & \Cdots & \mathbb{K}_{\Tilde{a}}^{[n-1],(M)}(x) + \frac{v_{M;\Tilde{a}}^R}{x-x_M}
    \end{vNiceMatrix},
    \end{equation*}
    we can decompose this determinantal formula as follows
    \begin{multline*}
        \check{A}^{(a)}_n(x) = \sum_{\Tilde{a}=1}^p \frac{R_{a,\Tilde{a}}(x)}{\tau_n} \sum_{i=0}^{n-1} A_i^{(\tilde{a})}(x)\begin{vNiceMatrix}
        \mathbb{D}_{n-M+1}^{(1)}- \mathbb{W}_{n-M+1}^{(1)}& \Cdots & \mathbb{D}_{n-M+1}^{(M)}- \mathbb{W}_{n-M+1}^{(M)}\\
        \Vdots & & \Vdots \\
        \mathbb{D}_{n-1}^{(1)}- \mathbb{W}_{n-1}^{(1)}& \Cdots & \mathbb{D}_{n-1}^{(M)}- \mathbb{W}_{n-1}^{(M)}\\[4pt]
        \mathbb{D}_{i}^{(1)}- \mathbb{W}_{i}^{(1)}& \Cdots & \mathbb{D}_{i}^{(M)}- \mathbb{W}_{i}^{(M)}
    \end{vNiceMatrix} \\
    + \sum_{\Tilde{a}=1}^p \frac{R_{a,\Tilde{a}}(x)}{\tau_n}\begin{vNiceMatrix}
        \mathbb{D}_{n-M+1}^{(1)}- \mathbb{W}_{n-M+1}^{(1)}& \Cdots & \mathbb{D}_{n-M+1}^{(M)}- \mathbb{W}_{n-M+1}^{(M)}\\
        \Vdots & & \Vdots \\
        \mathbb{D}_{n-1}^{(1)}- \mathbb{W}_{n-1}^{(1)}& \Cdots & \mathbb{D}_{n-1}^{(M)}- \mathbb{W}_{n-1}^{(M)}\\[4pt
        ]
        \frac{v_{1;\Tilde{a}}^R}{x-x_1} & \Cdots & \frac{v_{M;\Tilde{a}}^R}{x-x_M}
    \end{vNiceMatrix}.
    \end{multline*}
    The sum on $i$ is truncated when $i = n-M+1$ up to $n-1$. We then obtain Equation \eqref{EqCheckA}.
    
    In order to prove Equation \eqref{EqCheckB}, let us make use of the second equation in Proposition \ref{ConnectAB}. Entrywise, we find (for  $n \geq M$):
    \begin{align*}
        \check{B}_n^{(b)}(x) & = \Omega_{n.n-M}B_{n-M}^{(b)}(x) + \dots + \Omega_{n,n-1}B_{n-1}^{(b)}(x) + B_n^{(b)}(x) \\
        & = B_n^{(b)}(x) + \begin{bNiceMatrix}
            \Omega_{n.n-M} & \Cdots & \Omega_{n.n-1}
        \end{bNiceMatrix}\begin{bNiceMatrix}
            B_{n-M}^{(b)}(x) \\
            \Vdots \\
            B_{n-1}^{(b)}(x)
        \end{bNiceMatrix}.
    \end{align*}
    Using now Proposition \ref{OmegaComponents}, this can be rewritten as: 
    \begin{align*}
        \check{B}_n^{(b)}(x) =  B_n^{(b)}(x) - & \begin{bNiceMatrix}
            \mathbb{D}_n^{(1)}- \mathbb{W}_n^{(1)}& \Cdots & \mathbb{D}_n^{(M)}- \mathbb{W}_n^{(M)}
        \end{bNiceMatrix} & \begin{bNiceMatrix}
            \mathbb{D}_{n-M}^{(1)}- \mathbb{W}_{n-M}^{(1)}& \Cdots & \mathbb{D}_{n-M}^{(M)}- \mathbb{W}_{n-M}^{(M)}\\
            \Vdots & & \Vdots\\
            \mathbb{D}_{n-1}^{(1)}- \mathbb{W}_{n-1}^{(1)}& \Cdots & \mathbb{D}_{n-1}^{(M)}- \mathbb{W}_{n-1}^{(M)}
        \end{bNiceMatrix}^{-1}\begin{bNiceMatrix}
            B_{n-M}^{(b)}(x) \\
            \Vdots \\
            B_{n-1}^{(b)}(x)
        \end{bNiceMatrix}.
    \end{align*}
    And from this last relation it is straight forward to prove Equation \eqref{EqCheckB}.
\end{proof}

The formulas for calculating $\check{A}(x)$ can be computationally challenging to obtain if one only wishes to compute them for a specific $n$. However, if the $\mathbb{K}^{[n-1],(i)}(x)$ have already been calculated, obtaining the next CD kernel requires merely adding an additional term to the sum.

\begin{Proposition} \label{kbb reducido}
    The terms $\mathbb{K}_{a}^{[n],(i)}(x)$ can be re-expressed as follows:
    \begin{multline*}
        \mathbb{K}_{a}^{[n],(i)}(x) = \frac{1}{(x-x_i)} \left[ \sum_{j=1}^p \sum_{l=1}^j T_{n+l,n-p+j} A_{n+l}^{(a)}(x) \left( \mathbb{D}_{n+j-p}^{(i)} - \mathbb{W}_{n+j-p}^{(i)} - D_{n+j-p}(x)\boldsymbol{v}_i^R \right) \right. \\ - \left. \sum_{j=0}^{q-1} \sum_{l=0}^j T_{n-l,n+q-j} A_{n-l}^{(a)}(x) \left( \mathbb{D}_{n+q-j}^{(i)} - \mathbb{W}_{n+q-j}^{(i)} - D_{n+q-j}(x)\boldsymbol{v}_i^R \right)\right]. 
    \end{multline*}
\end{Proposition}

\begin{proof}
    This result is obtained by substituting the conclusions of Theorems \ref{KcomoT} and \ref{KDcomoT} into Definition \ref{K bb} and performing certain simplifications.
\end{proof}

With this, we obtain expressions for the CD kernels involving sums that do not grow with $n$, depending only on $q$ and $p$.

\begin{Remark} \label{Remark1}
    From the second equation in Proposition \ref{ConnectAB}, one could also study the relation for $n < M$: 
    \begin{equation*}
        \check{B}_n^{(b)}(x) = \Omega_{n,0}B_{0}^{(b)}(x) + \dots + \Omega_{n,n-1}B_{n-1}^{(b)}(x) + B_n^{(b)}(x).
    \end{equation*}
    However, the explicit formulas for $\Omega_{i,j}$ when $i < M$ cannot be found. Studying the n-th component of Equation \eqref{EqConnectD}, taking the limit when $x \rightarrow x_i$ and acting on the right with $\boldsymbol{v}_i^R$ yields: 
    \begin{equation*}
        \left( \Omega B(x_i)\xi(x_i)\boldsymbol{v}_i^LR'(x_i)\boldsymbol{v}_i^R \right)^{[n]} = \left( \Omega D(x_i)\boldsymbol{v}_i^R - \lim_{x\rightarrow x_i}  \Omega\int_\Delta B(y) \d \check{\mu}(y) \frac{R(y)\boldsymbol{v}_i^R}{x-y} \right)^{[n]},
    \end{equation*}
    and explicitly decomposing the perturbed measure: 
    \begin{equation*}
        \Omega\int_\Delta B(y) \d \mu(y) \frac{\boldsymbol{v}_i^R}{x_i-y}-B(x_i)\xi(x_i)\lim_{x\rightarrow x_i}\frac{\boldsymbol{v}_i^LR(x_i)\boldsymbol{v}_i^R}{y-x_i} = \Omega D(x_i)\boldsymbol{v}_i^R - \Omega B(x_i)\xi(x_i)\boldsymbol{v}_i^LR'(x_i)\boldsymbol{v}_i^R . 
    \end{equation*}
    The relation is satisfied trivially and we cannot obtain any information on the $\Omega_{i,j}$.

   Nevertheless, we can assure that the maximum degrees for the $\check{B}_n^{(b)}(x)$ are reached for $n<M$. Since the maximum degrees appear in terms of the form $\check{B}_{ip+j}^{(1+j)}(x)$, 
   \begin{equation*}
       \check{B}_{ip+j}^{(1+j)}(x) = \check{B}_{ip+j}^{(1+j)}(x) + \Omega_{ip+j,ip+j-1} \check{B}_{ip+j-1}^{(1+j)}(x) + \dots + \Omega_{ip+j,0}\check{B}_{0}^{(1+j)}(x).
   \end{equation*}
   All the terms multiplied by the $\Omega$ components are strictly of lower degree and $\check{B}_{ip+j}^{(1+j)}(x)$ reaches the desired degree.
\end{Remark}
\subsection{Eigenvalues of Arbitrary Multiplicity} \label{GeneralSpectra}

We now address the case of a perturbation matrix with arbitrary multiplicity. In this context, the most general measure is expressed as follows:

\begin{equation*}
    \d \check{\mu}(x) = \d \mu(x)R^{-1}(x) + \sum_{i=1}^M\sum_{j=1}^{s_i}\sum_{k=0}^{\kappa_{i,j}-1}\boldsymbol{\xi}_{i,j,k}(x) \left( \sum_{l=0}^k\frac{(-1)^l}{l!} \boldsymbol{v}_{i,j;k-l}^L\delta^{(l)}(x-x_i)\right),
\end{equation*}
where the first summation accounts for the different roots, the second and third summations account for the number of vectors in the canonical set of Jordan chains (associated with each eigenvalue), and the final summation correctly expands the action of each Jordan chain. The total number of distinct masses is $Np-r$. Many of the results derived in the simple eigenvalue scenario can be extended to this case, provided that the full theory of Jordan chains for $R(x)$ is applied. However, the different limits evaluated in the proofs of Proposition \ref{OmegaComponents} and Proposition \ref{Futuro Tau} require more careful consideration. Before proceeding, we establish a useful lemma.

\begin{Lemma} \label{LemmaEigenvectors}
    For a holomorphic matrix function $f(x)$ (with holomorphic entries) and a matrix polynomial $R(x)$, both of order $p$, the following holds:
    \begin{equation*}
        \sum_{l=0}^k \frac{1}{l!} \boldsymbol{v}_{i,j;k-l}^L \frac{\d^l}{\d x^l} \left[ R(x) f(x) \right]_{x=x_i} = 0,
    \end{equation*}
    where $\boldsymbol{v}_{i,j;k-l}$ is a Jordan vector and $k \leq \kappa_{i,j}-1$. The same result holds if we apply this expression on the left to the combination $f(x)R(x)$.
\end{Lemma}

\begin{proof}
    Differentiating $l$ times gives:
    \begin{equation*}
        \sum_{l=0}^k \frac{1}{l!} \boldsymbol{v}_{i,j;k-l}^L \sum_{m=0}^l \binom{l}{m}  R^{(m)}(x_i)f^{(l-m)}(x_i).
    \end{equation*}
    By changing the summation index as $\lambda = l-m$ and $\mu = m$, the expression becomes:
    \begin{equation*}
        \sum_{\lambda=0}^k \frac{1}{\lambda!} \left( \sum_{\mu=0}^{k-\lambda} \frac{1}{\mu!} \boldsymbol{v}_{i,j;k-\lambda-\mu}^L R^{(\mu)}(x_i) \right) f^{(\lambda)}(x_i).
    \end{equation*}
    This results in zero, as the term in parentheses represents the action of an appropriate Jordan chain on $R(x)$ and its derivatives.
\end{proof}

\begin{Proposition}
    The following limits:
    \begin{equation*}
\begin{aligned}
	        &\lim_{x \rightarrow x_i} \check{D}(x) R(x) \left( \sum_{l=0}^k \frac{1}{l!}\overleftarrow{\frac{\d^l}{\d x^l}} \boldsymbol{v}_{i,j;k-l}^R \right), & k&\in  \{ 0, \dots, \kappa_{i,j}-1 \},
\end{aligned}
    \end{equation*}
    exists and depends only on $R(x)$ and its derivatives, the canonical set of Jordan chains associated with the eigenvalue $x_i$ (both left and right), the matrix polynomial $\check{B}(x)$, and the vector functions $\boldsymbol{\xi}_{i,j,k}(x)$.
\end{Proposition}

\begin{proof}
    Explicitly, the limit is:
    \begin{multline*}
        \lim_{x \rightarrow x_i} \left[  \int_\Delta \check{B}(t) \d \mu(t) \frac{R^{-1}(t)R(x)}{x-t} \right.\\+ \int_{-\infty}^\infty\d t\sum_{\tilde{\imath}=1;\tilde{\imath} \neq i}^M\sum_{\tilde{\jmath}=1}^{s_{\tilde{\imath}}}\sum_{\tilde{k}=0}^{\kappa_{\tilde{\imath},\tilde{\jmath}}-1} \check{B}(t)\boldsymbol{\xi}_{\tilde{\imath},\tilde{\jmath},\tilde{k}}(t) \left( \sum_{\tilde{l}=0}^{\tilde{k}}\frac{(-1)^{\tilde{l}}}{\tilde{l}!} \boldsymbol{v}_{\tilde{\imath},\tilde{\jmath};\tilde{k}-\tilde{l}}^L\frac{\delta^{(\tilde{l})}(t-x_{\tilde{\imath}})}{x-t} R(x)\right)  \\
        \left. + \int_{-\infty}^\infty\d t\sum_{\tilde{\jmath}=1}^{s_{i}}\sum_{\tilde{k}=0}^{\kappa_{i,\tilde{\jmath}}-1}\check{B}(t)\boldsymbol{\xi}_{i,\tilde{\jmath},\tilde{k}}(t) \left( \sum_{\tilde{l}=0}^{\tilde{k}}\frac{(-1)^{\tilde{l}}}{\tilde{l}!} \boldsymbol{v}_{i,\tilde{\jmath};\tilde{k}-\tilde{l}}^L\frac{\delta^{(\tilde{l})}(t-x_i)}{x-t}R(x)\right)\right] \left( \sum_{l=0}^k \frac{1}{l!}\overleftarrow{\frac{\d^l}{\d x^l}} \boldsymbol{v}_{i,j;k-l}^R \right).
    \end{multline*}
    The first  integral term, as well as the summation term for $i \neq \tilde{\imath}$, vanishes because we have a holomorphic function at $x=x_i$, and we can apply Lemma \ref{LemmaEigenvectors}. The non-zero term simplifies as:
    \begin{multline*}
        \lim_{x \rightarrow x_i} \left(\sum_{\tilde{\jmath}=1}^{s_{i}}\sum_{\tilde{k}=0}^{\kappa_{i,\tilde{\jmath}}-1} \sum_{\tilde{l}=0}^{\tilde{k}}\sum_{m=0}^{\tilde{l}} \frac{1}{(\tilde{l}-m)!}\left[\check{B}(t)\boldsymbol{\xi}_{i,\tilde{\jmath},\tilde{k}}(t)\right]^{(\tilde{l}-m)}_{t=x_i}  \boldsymbol{v}_{i,\tilde{\jmath};\tilde{k}-\tilde{l}}^L\frac{R(x)}{(x-x_i)^{m+1}}\right) \left( \sum_{l=0}^k \frac{1}{l!}\overleftarrow{\frac{\d^l}{\d x^l} }\boldsymbol{v}_{i,j;k-l}^R \right).
    \end{multline*}
    Using the derivative product formula, the limit becomes:
    \begin{multline*}
        \lim_{x \rightarrow x_i} \left(\sum_{\tilde{\jmath},\tilde{k},\tilde{l},m} \sum_{l=0}^k \sum_{\lambda=0}^l \frac{1}{\lambda!(l-\lambda)!}\frac{(-1)^{\lambda}(m+\lambda)!}{m!(\tilde{l}-m)!}\left[\check{B}(t)\boldsymbol{\xi}_{i,\tilde{\jmath},\tilde{k}}(t)\right]^{(\tilde{l}-m)}_{t=x_i}  \boldsymbol{v}_{i,\tilde{\jmath};\tilde{k}-\tilde{l}}^L\frac{R^{(l-\lambda)}(x)}{(x-x_i)^{m+\lambda+1}}\boldsymbol{v}_{i,j;k-l}^R \right). 
    \end{multline*}
    Adjusting the summations on $l$ and $\lambda$ yields:
    \begin{multline*}
        \lim_{x \rightarrow x_i} \left(\sum_{\tilde{\jmath},\tilde{k},\tilde{l},m}  \sum_{\lambda=0}^k \sum_{\mu=0}^{k-\lambda} \frac{(-1)^{\lambda}(m+\lambda)!}{\lambda!m!(\tilde{l}-m)!}\left[\check{B}(t)\boldsymbol{\xi}_{i,\tilde{\jmath},\tilde{k}}(t)\right]^{(\tilde{l}-m)}_{t=x_i}  \boldsymbol{v}_{i,\tilde{\jmath};\tilde{k}-\tilde{l}}^L\frac{R^{(\mu)}(x)}{(x-x_i)^{m+\lambda+1}} \frac{1}{\mu ! }\boldsymbol{v}_{i,j;k-\lambda+\mu}^R \right).
    \end{multline*}
    By subtracting the following combination:
    \begin{equation*}
        \sum_{\tilde{\jmath},\tilde{k},\tilde{l},m,\lambda,\mu} \frac{(-1)^{\lambda}(m+\lambda)!}{\lambda!m!(\tilde{l}-m)!}\left[\check{B}(t)\boldsymbol{\xi}_{i,\tilde{\jmath},\tilde{k}}(t)\right]^{(\tilde{l}-m)}_{t=x_i}  \boldsymbol{v}_{i,\tilde{\jmath};\tilde{k}-\tilde{l}}^L\frac{1}{(x-x_i)^{m+\lambda+1}} \frac{1}{\mu ! }R^{(\mu)}(x_i)\boldsymbol{v}_{i,j;k-\lambda+\mu}^R,
    \end{equation*}
    which is always zero, the limit simplifies to:
    \begin{multline*}
         \sum_{\tilde{\jmath},\tilde{k},\tilde{l},m,\lambda,\mu} \frac{(-1)^{\lambda}(m+\lambda)!}{\lambda!\mu!m!(\tilde{l}-m)!}\left[\check{B}(t)\boldsymbol{\xi}_{i,\tilde{\jmath},\tilde{k}}(t)\right]^{(\tilde{l}-m)}_{t=x_i}  \lim_{x \rightarrow x_i} \left( \boldsymbol{v}_{i,\tilde{\jmath};\tilde{k}-\tilde{l}}^L\frac{R^{(\mu)}(x)-R^{(\mu)}(x_i)}{(x-x_i)^{m+\lambda+1}} \boldsymbol{v}_{i,j;k-\lambda+\mu}^R \right),
    \end{multline*}
    which only depends on the derivatives of $R(x)$, and the limit exists.
\end{proof}
In this section, we introduce additional notation to extend the Christoffel-type formulas to a more general setting. 

\begin{Definition}
We define the following expressions:
\begin{align*}
    \mathbb{D}_{n,j}^{(i)}& \coloneq \begin{bNiceMatrix}
        D_n(x_i)\boldsymbol{v}^R_{i,j;0} & \Cdots & \displaystyle\sum_{l=0}^{\kappa_{i,j}-1} \frac{1}{l!} D^{(l)}_n(x_i)\boldsymbol{v}^R_{i,j;\kappa_{i,j}-1-l}
    \end{bNiceMatrix}, \\
    \mathbb{W}_{n,j}^{(i)}& \coloneq \begin{multlined}[t][.75\textwidth]
    	\lim_{x \rightarrow x_i}
    \int_{-\infty}^\infty\d t	 \sum_{\tilde{\jmath}=1}^{s_i}\sum_{\tilde{k}=0}^{\kappa_{i,\tilde{\jmath}}-1}\sum_{\tilde{l}=0}^{\tilde{k}} \frac{(-1)^{\tilde{l}}}{\tilde{l}!} B(t)\boldsymbol{\xi}_{i,\tilde{\jmath},\tilde{k}}(t)  \boldsymbol{v}_{i,\tilde{\jmath};\tilde{k}-\tilde{l}}^L \frac{\delta^{(\tilde{l})}(t-x_i)R(x)}{(x-t)}\\
    \times	\begin{bNiceMatrix}
        \boldsymbol{v}^R_{i,j,0} & \Cdots &\displaystyle \sum_{l=0}^{\kappa_{i,j}-1} \frac{1}{l!} \overleftarrow{\dfrac{\d^l}{\d x^l}} \boldsymbol{v}^R_{i,j,\kappa_{i,j}-1-l}
    \end{bNiceMatrix}, 
    \end{multlined}
\end{align*}
both of which are row vectors of length $\kappa_{i,j}$. Additionally, we define the vector:
\begin{equation*}
    \mathbb{D}_n^{(i)}-\mathbb{W}_n^{(i)}\coloneq \begin{bNiceMatrix}
        \mathbb{D}_{n,1}^{(i)}-\mathbb{W}_{n,1}^{(i)}& \Cdots & \mathbb{D}_{n,s_i}^{(i)}-\mathbb{W}_{n,s_i}^{(i)}
    \end{bNiceMatrix},
\end{equation*}
which has length $K_i$ (where $K_i$ represents the multiplicity of $x=x_i$ as a zero of $\det R(x)$).
\end{Definition}

\begin{Definition}
We also introduce the vector:
\begin{equation*}
    \mathbb{K}^{[n-1],(i)}(x) \coloneq \sum_{r=0}^{n-1}A_r(x) \left( \mathbb{D}_r^{(i)}- \mathbb{W}_r^{(i)} \right),
\end{equation*}
and the following notations:
\begin{align*}
    \mathbb{P}_j^{(i)}(x) & \coloneq \begin{bNiceMatrix}
        \dfrac{\boldsymbol{v}_{i,j;0}}{x-x_i} & \Cdots & \displaystyle\sum_{l=0}^{\kappa_{i,j}-1} \frac{1}{l!} \frac{\d^l}{\d x^l} \left( \frac{1}{x-x_i} \right) \boldsymbol{v}^R_{i,j;\kappa_{i,j}-1-l}
    \end{bNiceMatrix}, \\
    \mathbb{P}^{(i)}(x) & \coloneq \begin{bNiceMatrix}
        \mathbb{P}_1^{(1)}(x) & \Cdots & \mathbb{P}_{s_i}^{(i)}(x)
    \end{bNiceMatrix}.
\end{align*}
Here, $\mathbb{P}_j^{(i)}(x)$ is of length $\kappa_{i,j}$, whereas $\mathbb{K}^{[n-1],(i)}(x)$ and $\mathbb{P}(x,x_i)$ have length $K_i$.
\end{Definition}

We are now in a position to present the Christoffel-type formulas for eigenvalues of arbitrary multiplicity. A noteworthy distinction from the simple eigenvalue case is that the number of unknowns is $Np - r$, which differs from $M$, the number of distinct roots of $\det R(x)$.

\begin{Definition}
Within this framework, the $\tau$-determinants are expressed as:
\begin{equation*}
    \tau_{n} \coloneq \begin{vNiceMatrix}
        \mathbb{D}_{n-Np+r+1}^{(1)}- \mathbb{W}_{n-Np+r+1}^{(1)}& \Cdots & \mathbb{D}_{n-Np+r+1}^{(M)}- \mathbb{W}_{n-Np+r+1}^{(M)}\\
        \Vdots & & \Vdots \\
        \mathbb{D}_{n}^{(1)}- \mathbb{W}_{n}^{(1)}& \Cdots & \mathbb{D}_{n}^{(M)}- \mathbb{W}_{n}^{(M)}
    \end{vNiceMatrix},
\end{equation*}
with $n \in \{ Np+r-1, Np+r, Np+r+1, \dots  \} $.
\end{Definition}

\begin{Theorem}
For a matrix polynomial satisfying the conditions outlined in Section \ref{CondicionesMatricesLideresFinales}, the relationship between the perturbed and original orthogonal polynomials for a Geronimus perturbation is given by the following formulas for $n \geq Np-r$:
\begin{align}
    \label{EqCheckA-AM} \check{A}^{(a)}_n(x) &= \sum_{\tilde{a}=1}^p \frac{R_{a,\tilde{a}}(x)}{\tau_n} \begin{vNiceMatrix}
        \mathbb{D}_{n-Np+r+1}^{(1)}- \mathbb{W}_{n-Np+r+1}^{(1)}& \Cdots & \mathbb{D}_{n-Np+r+1}^{(M)}- \mathbb{W}_{n-Np+r+1}^{(M)}\\
        \Vdots & & \Vdots \\
        \mathbb{D}_{n-1}^{(1)}- \mathbb{W}_{n-1}^{(1)}& \Cdots & \mathbb{D}_{n-1}^{(M)}- \mathbb{W}_{n-1}^{(M)}\\[4pt]
        \mathbb{K}_{\tilde{a}}^{[n-M],(1)}(x) + \mathbb{P}_{\tilde{a}}^{(1)}(x) & \Cdots & \mathbb{K}_{\tilde{a}}^{[n-M],(M)}(x) + \mathbb{P}_{\tilde{a}}^{(M)}(x)
    \end{vNiceMatrix}, \\
    \label{EqCheckB-AM} \check{B}^{(b)}_n(x) &= \frac{1}{\tau_{n-1}} \begin{vNiceMatrix}
        \mathbb{D}_{n-Np+r}^{(1)}- \mathbb{W}_{n-Np+r}^{(1)}& \Cdots & \mathbb{D}_{n-Np+r}^{(M)}- \mathbb{W}_{n-Np+r}^{(M)}& B_{n-Np+r}^{(b)}(x) \\
        \Vdots & & \Vdots & \Vdots \\
        \mathbb{D}_{n-1}^{(1)}- \mathbb{W}_{n-1}^{(1)}& \Cdots & \mathbb{D}_{n-1}^{(M)}- \mathbb{W}_{n-1}^{(M)}& B_{n-1}^{(b)}(x) \\[4pt]
        \mathbb{D}_{n}^{(1)}- \mathbb{W}_{n}^{(1)}& \Cdots & \mathbb{D}_{n}^{(M)}- \mathbb{W}_{n}^{(M)}& B_n^{(b)}(x)
    \end{vNiceMatrix}.
\end{align}
\end{Theorem}
\begin{proof}
    We will highlight the key distinction between these formulas and those previously established (Equation \eqref{EqCheckA} and \eqref{EqCheckB}), particularly focusing on Equation \eqref{EqCheckA-AM}. In the proof of Proposition \ref{Futuro Tau}, when dealing with different Jordan chains, it becomes necessary to consider the term
    \begin{equation*}
        \frac{R(x)-R(y)}{x-y}\left( \sum_{l=0}^k\frac{1}{l!} \overleftarrow{\frac{\d^l}{\d y^l}} \boldsymbol{v}^R_{i,j;k-1-l} \right)_{y=x_i},
    \end{equation*}
    which, by applying Lemma \ref{LemmaEigenvectors}, can be reformulated as:
    \begin{equation*}
        R(x)\left( \sum_{l=0}^k\frac{1}{l!} \frac{\d^l}{\d y^l}\left[ \frac{1}{x-y} \right]_{y=x_i} \boldsymbol{v}^R_{i,j;k-1-l} \right).
    \end{equation*}
    This expression is simply \( R(x) \) multiplied by the \( k \)-th component of the vector \( \mathbb{P}_j(x,x_i) \). Hence, the formulas generalize from 
    \begin{equation*}
        \mathbb{K}^{[n-1],(i)}(x) + \frac{\boldsymbol{v}_{i}}{x-x_i}, 
    \end{equation*}
    in the simple eigenvalue case, to 
    \begin{equation*}
        \mathbb{K}^{[n-1],(i)}(x) + \mathbb{P}^{(i)}(x),
    \end{equation*}
    in the more general case. 
\end{proof}

\begin{Remark}
    Let us now extend the proof of Theorem \ref{ExistsTheorem} to the present context. Proposition \ref{Futuro Tau} can be restated as follows:
    \begin{multline*}  
    	R(x)  \begin{bNiceMatrix}
    	\mathbb{K}^{[n-1],(1)}(x) + \mathbb{P}^{(1)}(x) & \Cdots & \mathbb{K}^{[n-1],(M)}(x) + \mathbb{P}^{(M)}(x)
    	\end{bNiceMatrix} \\ = \begin{bNiceMatrix}
    	    \check{A}_n^{(1)}(x)  & \Cdots & \check{A}_{n+Np-r-1}^{(1)}(x) \\
    	    \Vdots & & \Vdots \\
    	    \check{A}_n^{(p)}(x) & \Cdots & \check{A}_{n+Np-r-1}^{(p)}(x)
    	\end{bNiceMatrix} \begin{bNiceMatrix}
    	    \Omega_{n,n-Np+r} & \Cdots & & \Omega_{n,n-1} \\
    	    0 & & &   \\
    	    \Vdots & \Ddots & \Ddots[shorten-end=-8pt] & \Vdots \\
    	    0 & \Cdots & 0 & \Omega_{n+Np-r-1,n-1}
    	\end{bNiceMatrix} \\\times\begin{bNiceMatrix}
    	\mathbb{D}_{n-Np+r}^{(1)}- \mathbb{W}_{n-Np+r}^{(1)}& \Cdots & \mathbb{D}_{n-Np+r}^{(M)}- \mathbb{W}_{n-Np+r}^{(M)}\\
    	    \Vdots & & \Vdots \\
    	    \mathbb{D}_{n-1}^{(1)}- \mathbb{W}_{n-1}^{(1)}& \Cdots & \mathbb{D}_{n-1}^{(M)}- \mathbb{W}_{n-1}^{(M)}
    	\end{bNiceMatrix}.
    \end{multline*}
    Assuming that orthogonality holds even when \( \tau_{n-1} = 0 \), there exists a vector \( \boldsymbol{c} \) of size \( Np-r \), such that by studying points where \( \det R(x) \neq 0 \):
    \begin{equation*}
        \begin{bNiceMatrix}
    	\mathbb{K}^{[n-1],(1)}(x) + \mathbb{P}^{(1)}(x) & \Cdots & \mathbb{K}^{[n-1],(i)}(x) + \mathbb{P}^{(i)}(x) & \Cdots & \mathbb{K}^{[n-1],(M)}(x) + \mathbb{P}^{(M)}(x)
    	\end{bNiceMatrix} \boldsymbol{c} = 0.
    \end{equation*}
    Let us now examine \( \mathbb{P}^{(i)}(x) \) in closer detail. The vector contains derivatives of \( \frac{1}{x-y} \) evaluated at \( y = x_i \), which are poles whose multiplicity depends on the order of the derivative. Specifically, we have:
    \begin{equation*}
        \mathbb{P}^{(i)}(x) = \begin{bNiceMatrix}
            \frac{\boldsymbol{v}_{i,1;0}^R}{x-x_i} & \Cdots & \sum_{l=0}^{\kappa_{i,1}-2} \frac{\boldsymbol{v}_{i,1;\kappa_{i,1}-2-l}^R}{(x-x_i)^{l+1}} & \sum_{l=0}^{\kappa_{i,1}-1} \frac{\boldsymbol{v}_{i,1;\kappa_{i,1}-1-l}^R}{(x-x_i)^{l+1}} & \Cdots &\frac{\boldsymbol{v}_{i,s_i;0}^R}{x-x_i} & \Cdots & \sum_{l=0}^{\kappa_{i,s_i}-1} \frac{\boldsymbol{v}_{i,s_i;\kappa_{i,s_i}-1-l}^R}{(x-x_i)^{l+1}}
        \end{bNiceMatrix}.
    \end{equation*}
    Multiplying the expression by \( (x-x_i)^{\kappa_{i,j}-1} \), where \( \kappa_{i,j} \) is the largest partial multiplicity, we integrate the relation over the complex plane using a counterclockwise contour \( C_i \) around \( x_i \) (and not any other eigenvalue):
    \begin{equation*}
        \oint_{C_i} (z-x_i)^{\kappa_{i,j}-1} \begin{bNiceMatrix}
			\mathbb{K}^{[n-1],(1)}(z) + \mathbb{P}^{(1)}(z) & \Cdots & \mathbb{K}^{[n-1],(M)}(z) + \mathbb{P}^{(M)}(z)
		\end{bNiceMatrix} \boldsymbol{c} \: \d z.
    \end{equation*}
    The only nonzero terms will correspond to those involving poles of multiplicity \( \kappa_{i,j} \), which are associated with Jordan chains of length \( \kappa_{i,j} \). Without loss of generality, assume that the partial multiplicities for \( j \in \{ 1, \dots, s_i \} \) are the same. Then:
    \begin{equation*}
        2\pi \ii\begin{bNiceMatrix}
			0 & \Cdots & 0 & \boldsymbol{v}_{i,1;0}^R & 0 &\Cdots & 0 & \boldsymbol{v}_{i,s_i;0}^R & 0 & \Cdots[shorten-end=8pt] 
		\end{bNiceMatrix} \boldsymbol{c} = 0. 
    \end{equation*}
    The nonzero entries correspond to positions \( \kappa_{i,1}, \dots, \kappa_{i,s_i} \), we obtain:
    \begin{equation*}
        \boldsymbol{v}_{i,1;0}^R c_{\kappa_{i,1}} + \boldsymbol{v}_{i,2;0}^R c_{\kappa_{i,2}} + \dots + \boldsymbol{v}_{i,s_i;0}^R c_{\kappa_{i,s_i}} = 0.
    \end{equation*}
    However, since the vectors \( \boldsymbol{v}_{i,j;0}^R \) are linearly independent, it follows that the only possible conclusion is \( c_{\kappa_{i,j}} = 0 \). We then multiply the preceding expression by \( (x-x_i)^{\kappa_{i,j}-2} \). The terms that exhibit a pole of multiplicity \( \geq \kappa_{i,j}-1 \) are precisely those for which we have already established that \( c_i = 0 \), along with the terms immediately preceding them. At this juncture, we could also consider the Jordan chains associated with a partial multiplicity equal to \( \kappa_{i,j}-1 \), and the deduction would remain valid. Furthermore, the poles of multiplicity \( \kappa_{i,j}-1 \) will similarly correspond to the associated eigenvectors \( \boldsymbol{v}_{i,j;0}^R \), and the same reasoning applies as previously discussed. 

    Consequently, we would systematically deduce that all entries of \( \boldsymbol{c} \) that are multiplied by the vector \( \mathbb{K}^{[n-1],(i)}(z) + \mathbb{P}^{(i)}(x) \) must be zero. Extending this conclusion for \( i \in\{1, \dots, M\} \), we ultimately arrive at the result \( \boldsymbol{c} = 0 \). Thus, we encounter a contradiction once again.

\end{Remark}
\subsection{Left Multiplication}
We previously described Geronimus perturbations as a right multiplication of the perturbed measure. Here, we generalize the result by considering the case of a left multiplication, as follows:
\begin{equation*}
    L(x) \, \mathrm{d} \check{\mu}(x) = \mathrm{d} \mu(x),
\end{equation*}
where $L(x)$ is a matrix polynomial of degree $N$ given by
\[\begin{aligned}
    L(x) &= \sum_{l=0}^N L_l x^l, & L_l \in \mathbb{C}^{q \times q},
\end{aligned}\]
with its leading and subleading coefficients given by
\[\begin{aligned}
    L_{N} &= \begin{bmatrix}
        0_{r \times (q-r)} & 0_{r} \\
        I_{q-r} & 0_{(q-r) \times r}
    \end{bmatrix}, & 
    L_{N-1} &= \begin{bmatrix}
        \left[L^1_{N-1}\right]_{r \times (q-r)} & I_r \\
        \left[L^3_{N-1}\right]_{(q-r)} & \left[L^4_{N-1}\right]_{(q-r) \times r}
    \end{bmatrix}.
\end{aligned}\]
The determinantal polynomial has degree $M = Nq - r$, and it is required that $\sigma(L(x)) \cap \Delta = \emptyset$. For simplicity, we assume that $L(x)$ has simple eigenvalues, though this assumption can be generalized as discussed in \S \ref{GeneralSpectra}.

\begin{Proposition}
    The perturbed matrix of measures is expressed in terms of the original measure and the matrix polynomial as
    \begin{equation*}
        \mathrm{d} \check{\mu}(x) = L^{-1}(x) \, \mathrm{d} \mu(x) + \sum_{i=1}^M \boldsymbol{v}_i^R \boldsymbol{\xi}_i(x) \delta(x - x_i),
    \end{equation*}
    where $\delta$ denotes Dirac’s delta distribution, $\boldsymbol{v}_i^R$ are the right eigenvectors of $L(x)$ associated with simple eigenvalues, and $\boldsymbol{\xi}_i(x)$ is an arbitrary row vector function of size $p$.
\end{Proposition}

We present the results without detailed proofs, as the calculations are analogous to the case of right multiplication. To facilitate our analysis, we normalize the expressions as follows:
\[\begin{aligned}
    A(x) &= X_{[p]}^\top(x) \Bar{S}^\top, & B(x) &= H^{-1} S X_{[q]}(x),
\end{aligned}\]
where $A(x)$ is now monic.

\begin{Proposition}
    The connection matrix $\Omega$ can be expressed in two equivalent forms:
    \begin{equation*}
        \Omega = H^{-1} S L(\Lambda_{[q]}) \check{S}^{-1} \check{H} = \Bar{S}^{-\top} \check{\Bar{S}}^\top.
    \end{equation*}
    The connection matrix is upper unitriangular, with at most $M$ nonzero upper diagonals.
\end{Proposition}

\begin{Proposition}
    The following connection formulas hold:
\[    \begin{aligned}
        \Omega \check{B}(x) &= B(x) L(x), & A(x) \Omega &= \check{A}(x).
    \end{aligned}\]
\end{Proposition}

As before, these formulas do not directly allow us to solve for the components of the unknown $\Omega$. Therefore, we turn to studying the Cauchy transforms of these relations.

\begin{Proposition}
    In terms of Cauchy transforms, the following identities hold:
\[    \begin{aligned}
        \Omega \check{D}(x) &= D(x), & L(x) \check{C}(x) &= C(x)\Omega + \int_{\Delta} \frac{L(x) - L(y)}{x - y} \mathrm{d} \check{\mu}(y) \check{A}(y).
    \end{aligned}\]
\end{Proposition}

To state the Christoffel-type formulas in the context of a left Geronimus transformation, we introduce the following notation:

\begin{Definition}
    Based on Definitions \ref{D and W bb} and \ref{K bb}, we define:
    \begin{align*}
        \mathbb{C}_n^{(i)} &\coloneq \boldsymbol{v}_i^L C_n(x_i), \\
        \mathbb{W}_n^{(i)}&\coloneq \boldsymbol{v}_i^L L'(x_i) \boldsymbol{v}_i^R \boldsymbol{\xi}_i(x_i) A_n(x_i), \\
        \mathbb{K}^{[n],(i)}( y) &\coloneq \boldsymbol{v}_i^L K_C^{[n]}(x_i, y) - \boldsymbol{v}_i^L L'(x_i) \boldsymbol{v}_i^R \boldsymbol{\xi}_i(x_i) K^{[n]}(x_i, y).
    \end{align*}
\end{Definition}

\begin{Definition}
    The $\tau$-determinants are defined as follows:
    \begin{equation*}
        \tau_{n} \coloneq \begin{vNiceMatrix}
            \mathbb{C}_{n-M+1}^{(1)}- \mathbb{W}_{n-M+1}^{(1)}& \Cdots & \mathbb{C}_{n}^{(1)}- \mathbb{W}_{n}^{(1)}\\
            \Vdots & & \Vdots \\
            \mathbb{C}_{n-M+1}^{(M)}- \mathbb{W}_{n-M+1}^{(M)}& \Cdots & \mathbb{C}_{n}^{(M)}- \mathbb{W}_{n}^{(M)}
        \end{vNiceMatrix},
    \end{equation*}
    with $n \in \{ M-1, M, M-2, \dots \}$. 
\end{Definition}

The Christoffel-type formulas are derived by solving a linear system. Under this left perturbation, Theorem \ref{ExistsTheorem} holds, and if orthogonality exists, these $\tau$-determinants must be nonzero, which is equivalent to the fact that the linear system is inhomogeneous.

\begin{Theorem}
    The Christoffel-type formulas for a left Geronimus perturbation are:
    \begin{align*}
        \check{B}_n^{(b)}(y) &= \frac{1}{\tau_n} \sum_{\tilde{b}=1}^q \begin{vNiceMatrix}
            \mathbb{C}_{n-M+1}^{(1)}- \mathbb{W}_{n-M+1}^{(1)}& \Cdots & \mathbb{C}_{n-1}^{(1)}- \mathbb{W}_{n-1}^{(1)}& \mathbb{K}_{\tilde{b}}^{[n-M],(1)}( y) + \frac{v_{1; \tilde{b}}^L}{y - x_1} \\
            \Vdots & & \Vdots & \Vdots \\
            \mathbb{C}_{n-M+1}^{(M)}- \mathbb{W}_{n-M+1}^{(M)}& \Cdots & \mathbb{C}_{n-1}^{(M)}- \mathbb{W}_{n-1}^{(M)}& \mathbb{K}_{\tilde{b}}^{[n-M],(M)}(y) + \frac{v_{M; \tilde{b}}^L}{y - x_M}
        \end{vNiceMatrix} L_{\tilde{b}, b}(x), \\
        \check{A}_n^{(a)}(x) &= \frac{1}{\tau_{n-1}} \begin{vNiceMatrix}
            \mathbb{C}_{n-M}^{(1)}- \mathbb{W}_{n-M}^{(1)}& \Cdots & \mathbb{C}_n^{(1)}- \mathbb{W}_n^{(1)}\\
            \Vdots & & \Vdots \\
            \mathbb{C}_{n-M}^{(M)}- \mathbb{W}_{n-M}^{(M)}& \Cdots & \mathbb{C}_n(x_M) - \mathbb{W}_n^{(M)}\\[2pt]
            A_{n-M}^{(a)}(x) & \Cdots & A_n^{(a)}(x)
        \end{vNiceMatrix}.
    \end{align*}
\end{Theorem}
\subsection{Markov--Stieltjes Functions}

The Markov--Stieltjes function of the matrix of measures \(\d\mu\) is given by  
\[
F(z) = \int_{\Delta}\frac{\d\mu(x)}{z-x},
\]  
which is the Cauchy transform of the identity matrix and plays a crucial role in the spectral theory of orthogonal polynomials. In \cite{Zhe}, linear spectral transformations of the Markov--Stieltjes function for standard orthogonal polynomials were studied in the form  
\[
F(z) \longrightarrow \frac{A(z) F(z) + B(z)}{C(z)F(z) + D(z)},
\]  
where \(A, B, C,\) and \(D\) are polynomials.  

Zhedanov proved that Christoffel transformations generate a more general linear spectral transformation of the form  
\[
F(z) \longrightarrow A(z) F(z) + B(z).
\]  
He further proved that the composition of \(N\) successive Geronimus transformations generates a more general linear spectral transformation of the form  
\[
F(z) \longrightarrow \frac{F(z) + B(z)}{D(z)},
\]  
where \(\deg B = N-1\) and \(\deg D = N\).  

We now prove that similar statements hold for Geronimus transformations in this mixed multiple scenario.  

Here, we consider a Markov--Stieltjes \(q \times p\)-matrix function  
\[
F(z) \coloneq \int_{\Delta}\frac{\d\mu(x)}{z-x}.
\]  

\begin{Proposition}
	The Geronimus transformation described in Equation \eqref{eq:perturbed measure} amounts to a matrix linear spectral transformation of the form  
	\[
	\check{F}(z) = \big(F(z) + S(z)\big) R^{-1}(z),
	\]  
	where \(\deg S = \deg R - 1\) with matrix polynomial coefficients given by
	\[
\begin{aligned}
		S_k&=\int_{\Delta} \d\check \mu(x) (R_{k+1}+xR_{k+2}+\cdots+x^{N-k-1}R_N), & k&\in\{0,\dots,N-1\}.
\end{aligned}
		\]
\end{Proposition}

\begin{proof}
	This follows from the relationship \(\d \check{\mu}(x) R(x) = \d \mu(x)\). Indeed, we have the identity  
	\[
	\check{F}(z) R(z) = F(z) + S(z),
	\]  
	where  
	\[
	S(z) \coloneq \int_{\Delta} \d\check{\mu}(x) \frac{R(z) - R(x)}{z-x}.
	\]  
	Recalling \eqref{eq:Rxy}, we see that  
	\[
	\begin{aligned}
		S(z) &= \sum_{i=1}^N R_i \sum_{j=1}^i \left( \int_{\Delta} \d\check{\mu}(x) x^{i-j} \right) z^{j-1},
	\end{aligned}
	\]  
	which is a polynomial of degree \(N-1\).
\end{proof}

\subsection{Case Study: Non-Trivial Perturbations of  Jacobi--Piñeiro}
We will work with an explicit example. We consider Jacobi--Piñeiro multiple orthogonal polynomials, that is $q=1$, and in particular, $p=3$. Based on the results shown in \cite{ExamplesMOP}, we have $\d\mu_a=w_a\d\mu$, where the weight functions $w_a$ and measure $\d\mu$ are
\begin{align*}
&\begin{aligned}
 w_{a}(x;\alpha_a)=&x^{\alpha_a},& a&\in\{1,2,3\}, 
\end{aligned}&\d\mu(x)&=(1-x)^\beta\d x, & \Delta&=[0,1],
\end{align*}
with $\alpha_1,\alpha_2,\alpha_3,\beta>-1$ and, in order to have an AT system, $\alpha_i-\alpha_j\not\in\mathbb Z$ for $i\neq j$. An AT system of measures is such that the matrix 
\begin{equation*}
    \mathcal{A}_{n}(x) = \begin{bNiceMatrix}
        A_{3n}^{(1)}(x) & A_{3n+1}^{(1)}(x) & A_{3n+2}^{(1)}(x) \\[4pt]
        A_{3n}^{(2)}(x) & A_{3n+1}^{(2)}(x) & A_{3n+2}^{(2)}(x) \\[4pt]
        A_{3n}^{(3)}(x) & A_{3n+1}^{(3)}(x) & A_{3n+2}^{(3)}(x)
    \end{bNiceMatrix},
\end{equation*}
has terms of degree $n$ in the main diagonal and above it, below the main diagonal there are terms of degree $n-1$. If the orthogonality exists for a given vector of measures and the matrix of initial conditions, $\mathcal{A}_0$, is upper triangular with nonzero constants in the main diagonal and above it then we will have an AT system. 

The moments are 
\[\begin{aligned}
\label{MomentJP}
\int_{0}^{1}x^{\alpha_a+k}(1-x)^\beta\d x&=\dfrac{\Gamma(\beta+1)\Gamma(\alpha_a+k+1)}{\Gamma(\alpha_a+\beta+k+2)},
& k &\in \mathbb N_0.
\end{aligned}\]

    The Pochhammer symbols $(x)_n$, $x\in\C$ and $n\in\N_0$, are defined as,
\begin{align*}
 (x)_n\coloneq\dfrac{\Gamma(x+n)}{\Gamma(x)}=\begin{cases}
 x(x+1)\cdots(x+n-1)\;\text{if}\;n\in\N,\\
 1\;\text{if}\;n=0.
 \end{cases}
\end{align*}

The step-line in this situation means that we are considering triples of nonnegative integers, say $(n_1, n_2, n_3)$, of the form
\begin{align*}
	(0, 0, 0), (1, 0, 0), (1, 1, 0), (1, 1, 1), (2, 1, 1), (2, 2, 1), \dots.
\end{align*}
On the step-line, we have, for $m \in \mathbb{N}_0$,
\[
\begin{aligned}
	n &= 3m, & (n_1, n_2, n_3) &= (m, m, m), \\
	n &= 3m + 1, & (n_1, n_2, n_3) &= (m + 1, m, m), \\
	n &= 3m + 2, & (n_1, n_2, n_3) &= (m + 1, m + 1, m).
\end{aligned}
\]

    The Jacobi--Piñeiro polynomials of type I are, see  \cite{ExamplesMOP},
    \begin{align*}
\label{JPTypeI}
 A_n^{(a)}(x) \coloneq P^{(a)}_{n}(x;\alpha_1,\alpha_2,\alpha_3,\beta) =\sum_{l=0}^{n_a-1}C_{n}^{(a),l}x^l,
\end{align*}
with
\begin{multline*}
C_{n}^{(a),l}\coloneq\frac{
(-1)^{n-1}
\prod_{q=1}^{3}(\alpha_q+\beta+n)_{n_q}}{
(n_a-1)! 
\prod_{q=1,q\neq a}^{3}(\alpha_q-\alpha_a)_{n_q}} 
\frac{\Gamma(\alpha_a+\beta+n)}{\Gamma(\beta+n)\Gamma(\alpha_a+1)}
\\ \times
\frac{(-n_a+1)_l(\alpha_a+\beta+n)_l}{l!(\alpha_a+1)_l}
\prod_{q=1,q\neq a}^3
\frac{(\alpha_a-\alpha_q-n_q+1)_l}{(\alpha_a-\alpha_q+1)_l},
\end{multline*}
and 
\[\begin{aligned}
    n_a & = \deg(A_n^{(a)})+1, & n & = n_1+n_2+n_3.
\end{aligned}\]

    The monic Jacobi--Piñeiro polynomials of type II are, see  \cite{ExamplesMOP},
    \begin{align}
        B_n(x) \coloneq P_{n}(x;\alpha_1,\alpha_2,\alpha_3,\beta)=\sum_{l_1=0}^{n_1}\sum_{l_2=0}^{n_2}\sum_{l_3=0}^{n_3}
        C_{n}^{l_1,l_2,l_3}\, x^{l_1+l_2+l_3}
    \end{align}
    with
    \begin{multline*}
        C_{n}^{l_1,l_2,l_3}
        \coloneq
        (-1)^{n}\prod_{q=1}^3\dfrac{(\alpha_q+1)_{n_q}}{(\alpha_q+\beta+n+1)_{n_q}}\dfrac{(-n_q)_{l_q}}{l_q!}\dfrac{(\alpha_1+\beta+n_1+1)_{l_1+l_2+l_3}}{(\alpha_1+1)_{l_1+l_2+l_3}}
        \\
        \times\dfrac{(\alpha_1+n_1+1)_{l_2+l_3}(\alpha_{2}+n_{2}+1)_{l_3}}{(\alpha_1+\beta+n_1+1)_{l_2+l_3}(\alpha_{2}+\beta+n_1+n_{2}+1)_{l_3}}
        \\ \times \dfrac{(\alpha_2+\beta+n_1+n_2+1)_{l_2+l_3}(\alpha_{3}+\beta+n+1)_{l_3}}{(\alpha_2+1)_{l_2+l_3}(\alpha_{3}+1)_{l_3}}.
    \end{multline*}

The coefficients of the recurrence matrix $T$ (which is a banded (3,1) matrix) are given by (see \cite{ExamplesMOP}):
\begin{multline}
 T_{{pm+k}}^j=\dfrac{(\alpha_{a+1}+\beta+4m+k+1-j)}{\prod_{q=5+k-j}^{4+k}(\alpha_a+\beta+4m+k-j)}
 \\
 \times
 \dfrac{(\beta+3m+k+1-j)_j\prod_{q=1}^3(\alpha_a+\beta+3m+k+1-j)_{j}}{\prod_{q=k+1}^{3+k}(\alpha_a+\beta+4m+k+1-j)_{j}}
 \\
 \times
 \sum_{i=k+1}^{4+k-j}\dfrac{(\alpha_i+m)}{(\alpha_i+\beta+4m+k-j)_{j+2}}\dfrac{\prod_{q=1}^3(\alpha_i-\alpha_a+m)}{\prod_{q=k+1,q\neq i}^{4+k-j}(\alpha_i-\alpha_a)},
\end{multline}
for $j\in\{0,1,2,3\}$, $m\geq0$ and $k\in\{0,1,2\}$. The superscript $j=0$ corresponds to the only upper diagonal, whereas $j\in\{1,2,3\}$ corresponds to the first, second and third subdiagonals, respectively.

We are going to perturb the measure as follows:
\begin{equation*}
    \d \check{\mu}(x) \begin{bNiceMatrix}
        0 & x & 0 \\
        0 & 0 & 1-x\\
        1 & 0 & 0
    \end{bNiceMatrix} = \d \mu(x).
\end{equation*}
The perturbation matrix, which we will denote by $R(x)$, satisfies Conditions \ref{LeadingMatrixConditions} (as mentioned earlier Condition \ref{CondicionesMatricesLideresFinales} is merely a normalization condition). The perturbed measure can be expressed in terms of the original measure as follows:
\begin{align*}
    \d \check{\mu}(x) & = \frac{1}{x(1-x)}\d \mu(x) \begin{bNiceMatrix}
        0 & 0 & x(1-x) \\
        1-x & 0 & 0 \\
        0 & x & 0
    \end{bNiceMatrix} + \xi_0 \delta(x) \boldsymbol{e_1} + \xi_1 \delta(x-1) \boldsymbol{e}_2 \\
    & = \begin{bNiceMatrix}
        \dfrac{\d \mu_2(x)}{x} + \xi_0 \delta(x) & \dfrac{\d \mu_3(x)}{1-x}+ \xi_1 \delta(x-1) & \d \mu_1(x) 
    \end{bNiceMatrix} 
    \\
    & = \begin{bNiceMatrix}
    	x^{\alpha_2-1}(1-x)^\beta\d x+ \xi_0 \delta(x) & x^{\alpha_3}(1-x)^{\beta-1}+ \xi_1 \delta(x-1) & x^{\alpha_1}(1-x)^\beta\d x
    \end{bNiceMatrix} . 
\end{align*}

In what follows, let us assume that $\alpha_2,\beta > 0$. Note that, even if the two masses located at $x=0$ and $x=1$ were equal to zero, the new family of orthogonal polynomials would not correspond to Jacobi--Piñeiro multiple orthogonal polynomials for any arrangement of $(\alpha_1,\alpha_2,\alpha_3,\beta)$.   

    The Jacobi--Piñeiro multiple orthogonal polynomials at the endpoints of the interval are:
\[    \begin{aligned}
        P_n(0) & = (-1)^n \prod_{a=1}^3\frac{(\alpha_a+1)_{n_a}}{(\alpha_a+\beta+1)_{n_a}}, & P_n(1) & =  \frac{(\beta+1)_{n}}{\prod_{a=1}^3(\alpha_a+\beta+n+1)_{n_a}}.
    \end{aligned}\]

\begin{Proposition}
    The Cauchy transform for the type II Jacobi--Piñeiro can be explicitly computed for $x=0$, in particular we have:
    \begin{align*}
        D_n^{(2)}(0;\alpha_1,\alpha_2,\alpha_3,\beta) & = - \Gamma(\beta+1) \sum_{l_1=0}^{n_1}\sum_{l_2=0}^{n_2}\sum_{l_3=0}^{n_3}
        \frac{C_{n}^{l_1,l_2,l_3}}{(\alpha_2+l_1+l_2+l_3)_{\beta+1}}, \\
        D_n^{(3)}(1;\alpha_1,\alpha_2,\alpha_3,\beta) & = \Gamma(\beta) \sum_{l_1=0}^{n_1}\sum_{l_2=0}^{n_2}\sum_{l_3=0}^{n_3}
        \frac{C_{n}^{l_1,l_2,l_3}}{(\alpha_3+l_1+l_2+l_3)_{\beta}}.
    \end{align*}
\end{Proposition}
\begin{proof}
    Let us study the following integral, 
    \begin{equation*}
        \int_0^1 x^{l_1+l_2+l_3+\alpha_2-1}(1-x)^\beta \d x = \dfrac{\Gamma(\beta+1)\Gamma(\alpha_2+l_1+l_2+l_3)}{\Gamma(\alpha_2+\beta+l_1+l_2+l_3+1)} = \dfrac{\Gamma(\beta+1)}{(\alpha_2+l_1+l_2+l_3)_{\beta+1}}, 
    \end{equation*}
    which is always valid, provided that $\alpha_2 > 0$. The Cauchy transform of the type II Jacobi--Piñeiro can be expressed as:
    \begin{align*}
        D_n^{(2)}(0;\alpha_1,\alpha_2,\alpha_3,\beta) & = - \int_0^1P_n(x;\alpha_1,\alpha_2,\alpha_3,\beta) \frac{\d \mu_2(x)}{x} = - \sum_{l_1=0}^{n_1}\sum_{l_2=0}^{n_2}\sum_{l_3=0}^{n_3}
        C_{n}^{l_1,l_2,l_3}\, \int_\Delta x^{l_1+l_2+l_3} \frac{\d \mu_2(x)}{x} \\
        & = - \sum_{l_1=0}^{n_1}\sum_{l_2=0}^{n_2}\sum_{l_3=0}^{n_3}
        C_{n}^{l_1,l_2,l_3}\, \int_\Delta x^{l_1+l_2+l_3+\alpha_2-1}(1-x)^\beta \d x.
    \end{align*}
    Substituting the first result in this equation yields the desired relation. To prove the second relation, one proceeds similarly for $D_n^{(3)}(1;\alpha_1,\alpha_2,\alpha_3,\beta)$. 
\end{proof}

    For this example, the $\tau$-determinants take the form:
    \begin{equation*}
        \tau_{n} = \begin{vNiceMatrix}
            D_{n-1}^{(2)}(0)-P_{n-1}(0)\xi_0 & D_{n-1}^{(3)}(1)-P_{n-1}(1)\xi_1 \\[3pt]
            D_{n}^{(2)}(0)-P_{n}(0)\xi_0 & D_{n}^{(3)}(1)-P_{n}(1)\xi_1
        \end{vNiceMatrix}.
    \end{equation*}

\begin{Proposition}
    For $n > 1$, the Christoffel type formula between the type I Jacobi--Piñeiro polynomials and the Geronimus  perturbed polynomials is:
    \begin{multline*}
        \begin{bNiceMatrix}
            \check{P}_{n}^{(1)}(x) \\
            \check{P}_{n}^{(2)}(x) \\
            \check{P}_{n}^{(3)}(x) 
        \end{bNiceMatrix} = \frac{-1}{\tau_n}\sum_{i=0}^{n-2} \begin{bNiceMatrix}
            x P_i^{(2)}(x) \\
            (1-x) P_i^{(3)}(x) \\
            P_i^{(1)}(x)
        \end{bNiceMatrix} \begin{vNiceMatrix}
            D_i^{(2)}(0)-P_i(0)\xi_0 & D_i^{(3)}(1)-P_i(1)\xi_1 \\[3pt]
            D_{n-1}^{(2)}(0)-P_{n-1}(0)\xi_0 & D_{n-1}^{(3)}(1)-P_{n-1}(1)\xi_1
        \end{vNiceMatrix} \\ + \frac{1}{\tau_n}\begin{bNiceMatrix}
            D_{n-1}^{(3)}(1)-P_{n-1}(1)\xi_1 \\
            -D_{n-1}^{(2)}(0)+P_{n-1}(0)\xi_0 \\
            0
        \end{bNiceMatrix}.
    \end{multline*}
        For $n > 1$,  the Christoffel type formula between the type II Jacobi--Piñeiro polynomials and the Geronimus  perturbed polynomials is:
    \begin{equation*}
        \check{P}_n(x) = \frac{1}{\tau_{n-1}}\begin{vNiceMatrix}
            D_{n-2}^{(2)}(0)-P_{n-2}(0)\xi_0 & D_{n-2}^{(3)}(1)-P_{n-2}(1)\xi_1 & P_{n-2}(x) \\[3pt]
            D_{n-1}^{(2)}(0)-P_{n-1}(0)\xi_0 & D_{n-1}^{(3)}(1)-P_{n-1}(1)\xi_1 & P_{n-1}(x) \\[3pt]
            D_{n}^{(2)}(0)-P_{n}(0)\xi_0 & D_{n}^{(3)}(1)-P_{n}(1)\xi_1 & P_{n}(x)
        \end{vNiceMatrix}.
    \end{equation*}
\end{Proposition}
\begin{proof}
    It is direct consequence of Theorem \ref{ExplicitFormulas}, where we have substituted the corresponding perturbation matrix. 
\end{proof}
For an alternative expression using the 5-term recurrence relation we need
\begin{Proposition}
    The following relations hold:
    \begin{align*}
        \mathbb{K}_a^{[n],(0)} &= \begin{multlined}[t][.8\textwidth]
        \frac{1}{x} \left[ \sum_{j=1}^{3}\sum_{l=1}^j T^{j}_{n+l}P^{(a)}_{n+l}(x)\left( D^{(2)}_{n-j}(0)-\xi_0P_{n-j}(0)-D_{n-j}^{(2)}(x)  \right) \right. \\ \left. - T^{0}_n P^{(a)}_n(x)\left( D^{(2)}_{n+1}(0)-\xi_0P_{n+1}(0)-D_{n+1}^{(2)}(x)  \right) \right], 
        \end{multlined}\\
         \mathbb{K}_a^{[n],(1)} &= \begin{multlined}[t][.8\textwidth] \frac{1}{(x-1)} \left[ \sum_{j=1}^{3}\sum_{l=1}^j T^{j}_{n+l}P^{(a)}_{n+l}(x)\left( D^{(3)}_{n-j}(1)-\xi_1P_{n-j}(1)-D_{n-j}^{(3)}(x)  \right) \right. \\ \left. - T^{0}_n P^{(a)}_n(x)\left( D^{(3)}_{n+1}(1)-\xi_1P_{n+1}(1)-D_{n+1}^{(3)}(x)  \right) \right].  \end{multlined}
    \end{align*}
\end{Proposition}
Then the Christoffel formulas read
\begin{Proposition}
    For $n > 1$, the Christoffel type formula for $\check{P}^{(a)}_n(x)$ can be restated as: 
    \begin{multline*}
        \begin{bNiceMatrix}
            \check{P}_{n}^{(1)}(x) \\
            \check{P}_{n}^{(2)}(x) \\
            \check{P}_{n}^{(3)}(x) 
        \end{bNiceMatrix} = \frac{1}{\tau_n}\sum_{i=1}^{3} \sum_{j=1}^i T^j_{n+j} \begin{bNiceMatrix}
            x P_{n+j}^{(2)}(x) \\
            (1-x) P_{n+j}^{(3)}(x) \\
            P_{n+j}^{(1)}(x)
        \end{bNiceMatrix} \\ \times \begin{vNiceMatrix}
            D_{n-1}^{(2)}(0)-P_{n-1}(0)\xi_0 & D_{n-1}^{(3)}(1)-P_{n-1}(1)\xi_1 \\[3pt]
            \frac{1}{x}(D_{n-i}^{(2)}(0)-P_{n-i}(0)\xi_0 - D_{n-i}^{(2)}(x)) & \frac{1}{x-1}(D_{n-i}^{(3)}(1)-P_{n-i}(1)\xi_1 - D_{n-i}^{(3)}(x))
        \end{vNiceMatrix} \\[3pt] + \frac{1}{\tau_n} T^0_n \begin{bNiceMatrix}
            x P_{n+1}^{(2)}(x) \\
            (1-x) P_{n+1}^{(3)}(x) \\
            P_{n+1}^{(1)}(x)
        \end{bNiceMatrix}\begin{vNiceMatrix}
            D_{n-1}^{(2)}(0)-P_{n-1}(0)\xi_0 & D_{n-1}^{(3)}(1)-P_{n-1}(1)\xi_1 \\[3pt]
            D_{n+1}^{(2)}(0)-P_{n+1}(0)\xi_0 - D_{n+1}^{(2)}(x) & D_{n+1}^{(3)}(1)-P_{n+1}(1)\xi_1 - D_{n+1}^{(3)}(x)
        \end{vNiceMatrix} \\ + \frac{1}{\tau_n}\begin{bNiceMatrix}
            D_{n-1}^{(3)}(1)-P_{n-1}(1)\xi_1 \\
            -D_{n-1}^{(2)}(0)+P_{n-1}(0)\xi_0 \\
            0
        \end{bNiceMatrix}.
    \end{multline*}
\end{Proposition}

\section{On the Existence of Perturbed Orthogonality}
Previously, in Theorem \ref{ExistsTheorem}, we proved that if perturbed orthogonality exists, then the \(\tau\)-determinants are nonzero, leading to explicit Christoffel-type formulas for the perturbed polynomials. In this section, we will examine the reverse implication:

\begin{Theorem} \label{Non-cancellation}
	Perturbed orthogonality exists if the \(\tau\)-determinants are nonzero and there exist \(Np-r\) vector polynomials, \(\check{B}_n(x)\), satisfying: 
	\[
	\begin{aligned}
		\int_\Delta \sum_{b=1}^q \check{B}_n^{(b)}(x) \, \mathrm{d}\check{\mu}_{b,a}(x) \, x^l &= 0, & a &\in \{1,\dots,p\}, & l &\in \left\{0,\dots, \left\lceil\frac{n-a+2}{p}\right\rceil-1\right\},
	\end{aligned}
	\]
	for \(n \in \{0, \dots, Np-r-1\}\).
\end{Theorem}

\begin{proof}
   Define two new families of matrix polynomials:
   \[
\begin{aligned}
	   \tilde{A}(x) &= R(x)A(x), & \tilde{A}(x) &= \check{A}(x)\Omega,
\end{aligned}
   \]
   where \(R(x)\) is a polynomial matrix satisfying Condition \eqref{CondicionesMatricesLideresFinales} and \(\Omega\) is a unitriangular matrix with entries given by Equation \eqref{EqOmegaExplicit}. Since \(\Omega\) is invertible and the \(\tau\)-determinants are nonzero, it follows that the matrix is invertible. Note that no assumptions are made regarding the structure of these matrix families. As \(\sigma(R) \cap \Delta = \varnothing\), we obtain:
   \[
   S^{-1} = \int_\Delta X_{[q]}(x) \, \mathrm{d}\mu(x) \, A(x) = \int_\Delta X_{[q]}(x) \, \mathrm{d}\mu(x) \, R^{-1}(x) \check{A}(x)\Omega.
   \]
   The expression:
   \[
   X_{[q]}(x) \left[ \sum_{i=1}^{M} \sum_{j=1}^{s_i} \sum_{k=0}^{\kappa_{i,j}-1} \boldsymbol{\xi}_{i,j,k}(x) \left( \sum_{l=0}^k \frac{(-1)^l}{l!} \boldsymbol{v}_{i,j;k-l}^L \delta^{(l)}(x-x_i) \right) \right] \check{A}(x)\Omega,
   \]
   is strictly zero. Therefore:
   \[
   \begin{aligned}
   	S^{-1} &= \begin{multlined}[t][.8\textwidth]
   		\int_\Delta X_{[q]}(x) \, \mathrm{d}\mu(x) \, R^{-1}(x) \check{A}(x)\Omega \\
    + X_{[q]}(x) \left[ \sum_{i,j,k} \boldsymbol{\xi}_{i,j,k}(x) \left( \sum_{l=0}^k \frac{(-1)^l}{l!} \boldsymbol{v}_{i,j;k-l}^L \delta^{(l)}(x-x_i) \right) \right] \check{A}(x)\Omega 
   	\end{multlined}\\
   	&= \int_\Delta X_{[q]}(x) \, \mathrm{d}\check{\mu}(x) \, \check{A}(x)\Omega.
   \end{aligned}
   \]
   Define the lower unitriangular matrix \(\check{S} \coloneq \Omega S\) and the polynomial family \(\check{B}(x) \coloneq \check{S} X_{[q]}(x)\), so the last expression becomes:
   \[
   \int_\Delta \check{B}(x) \, \mathrm{d}\check{\mu}(x) \, \check{A}(x) = I.
   \]
   What has been proven so far is the biorthogonality condition for the families of matrix polynomials $\check{A}(x)$ and $B(x)$ as well as the degree structure of $\check{B}(x)$. The theorem's hypothesis concerning the orthogonality relations of $\check{B}_i(x)$ (for $i<{Np-r}$), under the biorthogonality condition, can be equivalently reformulated as an assumption about the degree structure of $\check{A}(x)$ as follows. For \(n \in \{0, \dots, N-2\}\), the block
   \[
   \mathcal{\check{A}}_n(x) =\coloneq
   \begin{bNiceMatrix}
   	\check{A}_{np}^{(1)}(x) & \Cdots & \check{A}_{np+p-1}^{(1)}(x) \\ 
   	\Vdots & & \Vdots \\ 
   	\check{A}_{np}^{(p)}(x) & \Cdots & \check{A}_{np+p-1}^{(p)}(x)
   \end{bNiceMatrix},
   \]
   has diagonal terms of degree \(n\), terms above the main diagonal  of degree at most \(n\), and lower diagonal terms of degree at most \(n-1\). For \(\mathcal{\check{A}}_{N-1}(x)\), we have: 
   \[
   \mathcal{\check{A}}_{N-1}(x) = \begin{bNiceMatrix}[first-row,last-row,nullify-dots]
   	| & \Cdots & \text{known} & \Cdots & | & & & & & \\
   	\check{A}_{(N-1)p}^{(1)}(x) & &\Cdots & &\check{A}_{Np-r-1}^{(1)}(x) & \check{A}_{Np-r}^{(1)}(x) & & \Cdots & &\check{A}_{Np-1}^{(1)}(x) \\ 
   	\Vdots & & & & \Vdots & \Vdots & & & & \Vdots \\ 
   	\check{A}_{(N-1)p}^{(p)}(x) & & \Cdots & & \check{A}_{Np-r-1}^{(p)}(x) & \check{A}_{Np-r}^{(p)}(x) & & \Cdots & & \check{A}_{Np-1}^{(p)}(x) \\
   	& & & & & | & \Cdots & \text{unknown} & \Cdots & |    
   \end{bNiceMatrix}.
   \]
   Namely, the known structure extends to the first $p-r-1$ columns (this column will have entries $\check{A}_{Np-r-1}^{(a)}$. However, no degree structure is assumed for $\check{A}_{n}(x)$ with $n \geq {Np-r}$. 
   
   Let us examine the relation \(\check{A}(x)\Omega = R(x)A(x)\) and focus on the first block of size \(p \times p\):
   \begin{equation} \label{Eqdegree}
   	\mathcal{\check{A}}_0(x)\left[ \Omega_{0,0} \right]_p + \mathcal{\check{A}}_1(x)\left[ \Omega_{1,0} \right]_p + \dots + \mathcal{\check{A}}_{N-1}(x)\left[ \Omega_{N-1,0} \right]_p + \mathcal{\check{A}}_N(x)\left[ \Omega_{N,0} \right]_p = R(x)\mathcal{A}_0.
   \end{equation}

    By construction, \(\left[ \Omega_{N,0} \right]_p\) and \(\left[ \Omega_{N-1,0} \right]_p\) are of the form:
	    \[
\begin{aligned}
	\left[\Omega_{N-1,0}\right]_p & = \begin{bNiceMatrix}
		\Omega_{Np-p,0} & \Cdots & & & & \Omega_{Np-p,p-1} \\
		\Vdots & & & & & \Vdots \\
		\Omega_{Np-r,0}	& & & & & 	 \\
		0 && & & &   \\
		\Vdots & \Ddots & \Ddots[shorten-end=-7pt]&&&\\
		0 & \Cdots & 0 & \Omega_{Np-1,r-1} & \Cdots & \Omega_{Np-1,p-1}
	\end{bNiceMatrix}, \\ \left[\Omega_{N,0}\right]_p & = \begin{bNiceMatrix}
		0 & \Cdots & 0 & \Omega_{Np,r} & \Cdots & \Omega_{Np,p-1} \\
		\Vdots & & & \Ddots & \Ddots[shorten-end=-5pt] & \Vdots \\
		& & & & & \Omega_{Np+p-r-1,p-1} \\[10pt]
		& & & & & 0 \\
		& & & & & \Vdots \\
		0 & & \Cdots & & & 0
	\end{bNiceMatrix},
\end{aligned}
\]
    with entries of the form $\Omega_{i+Np-r,i}$ nonzero due to the hypothesis that $\tau_n \neq 0$. The product \( R(x)\mathcal{A}_0 \) is a matrix polynomial as described in Condition \ref{LeadingMatrixConditions}. We focus our attention on terms of degree $N-1$. By hypothesis, it holds that
    \[
    \mathcal{\check{A}}_0(x)\left[ \Omega_{0,0} \right]_p + \mathcal{\check{A}}_1(x)\left[ \Omega_{1,0} \right]_p + \dots + \mathcal{\check{A}}_{N-2}(x)\left[ \Omega_{N-2,0} \right]_p = O(x^{N-2}).
    \]
    
    Considering the first column of Equation \eqref{Eqdegree} and focusing on the terms of degree \( N-1 \), we have:
    \begin{equation*}
        \begin{bNiceMatrix}
            \Omega_{Np-p,0}\check{A}_{Np-p}^{(1)}(x) + \dots + \Omega_{Np-r,0}\check{A}_{Np-r}^{(1)}(x) \\
            \Vdots \\
            \Omega_{Np-p,0}\check{A}_{Np-p}^{(p-r)}(x) + \dots + \Omega_{Np-r,0}\check{A}_{Np-r}^{(p-r)}(x) \\[7pt]
            \Omega_{Np-p,0}\check{A}_{Np-p}^{(p-r+1)}(x) + \dots + \Omega_{Np-r,0}\check{A}_{Np-r}^{(p-r+1)}(x) \\[7pt]
            \Omega_{Np-p,0}\check{A}_{Np-p}^{(p-r+2)}(x) + \dots + \Omega_{Np-r,0}\check{A}_{Np-r}^{(p-r+2)}(x) \\
            \Vdots \\
            \Omega_{Np-p,0}\check{A}_{Np-p}^{(p-r+1)}(x) + \dots + \Omega_{Np-r,0}\check{A}_{Np-r}^{(p-r+1)}(x)
        \end{bNiceMatrix} = \begin{bNiceMatrix}
            * \\
            \Vdots \\
            * \\[7pt]
            A_0^{(1)} \\[7pt]
            0 \\
            \Vdots \\
            0
        \end{bNiceMatrix} x^{N-1} + O(x^{N-2}),
    \end{equation*}
    where the \( * \) notation denotes constant terms, which could be zero. The product $\mathcal{\check{A}}_N(x)\left[ \Omega_{N,0} \right]_p$ does not appear in the first column, since $\left[ \Omega_{N,0} \right]_p$ has the first $r$ columns equal to zero. 
    
    It follows that the terms \( \check{A}_{Np-r}^{(a)} \) with \( a \in\{1,\dots, p-r \}\) can have degree up to \( N-1 \), as other terms of degree \( N-1 \) also appear in the summation. 
    For the case \( a = p-r+1 \), since the terms of the form $A_n^{(p-r+1)}$ are of degree \( N-2 \) for $n < Np-r$ and \( \check{A}_0^{(1)} \neq 0 \), the term \( \check{A}_{Np-r}^{(p-r+1)} \) must have degree \( N-1 \). Below these terms, for \( p-r+1 < a \leq p \), there are no terms of degree \( N-1 \) on the left-hand side, implying that \( \check{A}_{Np-r}^{(a)} \) can have degree up to \( N-2 \). Repeating this analysis \( r \) times for subsequent columns of Equation \eqref{Eqdegree}, completes the degree structure for \( \mathcal{\check{A}}_{N-1} \).

    Next, let us examine the \( r+1 \)-th column and focus on the terms of degree \( N \):

    \begin{equation*}
        \Omega_{Np,r} 
        \begin{bNiceMatrix}
            \check{A}_{Np}^{(1)}(x) \\
            \Vdots \\
            \check{A}_{Np}^{(p)}(x)
        \end{bNiceMatrix} = \begin{bNiceMatrix}
            A_r^{(r+1)} \\
            0 \\
            \Vdots \\
            0
        \end{bNiceMatrix} x^N + O(x^{N-1}).
    \end{equation*}

   Clearly, \( \check{A}_{Np}^{(1)}(x) \) must have degree \( N \), whereas the remaining terms can have a degree of at most \( N-1 \). By iterating this reasoning as described, we can conclude that up to \( n = Np+p-r \), all vector polynomials possess the required degree structure.
   
   Thus, the proof is complete. Under the assumption \( \tau_n \neq 0 \) for \( n \in \mathbb{N} \) and the existence of the orthogonality conditions for $B_i(x)$, we derive certain orthogonality relations for \( \check{A}(x) \) as well as their degree structure.
   \end{proof}
   
   Observe that assuming the orthogonality exists up to \( n = {Np-r}-1 \) does not guarantee its existence for all degrees, as it may truncate at some point. However, we have demonstrated that, provided the \( \tau \)-determinants are nonzero, the orthogonality does not truncate.
   
   It is now evident that the existence of perturbed orthogonality, provided that the $\tau$-determinants are nonzero, is related to the formulas for $\check{A}(x)$ and $\check{B}(x)$ in the case where $n < {Np-r}$. Let us examine the case for $n < {Np-r}$ in greater detail. The connection formulas for $\check{B}(x)$, as presented in Proposition \ref{ConnectAB}, are given by:
   \begin{equation} \label{EqcheckBB}
   \check{B}^{(b)}_n(x) = B^{(b)}_n(x) + \Omega_{n,n-1}B^{(b)}_{n-1}(x) + \cdots + \Omega_{n,0}B^{(b)}_0(x),
   \end{equation}
   and, as noted in Remark \ref{Remark1}, the method employed throughout this work does not provide explicit information about the $\Omega$ components when $n < {Np-r}$. However, under the assumption of the existence of perturbed orthogonality, the following relation holds:
\begin{equation}\label{EqOrthogo}
\begin{aligned}%
   \int_\Delta \check{B}_n(x) \, \mathrm{d}\check{\mu}(x) \, \left( X_{[p]}^\top(x) \boldsymbol{e}_l \right) = 0, \quad l \in \{0, \dots, n-1\}.
   \end{aligned}
\end{equation}
   
   \begin{Definition}
   We introduce the following notation:
   \begin{equation}\label{eq:I}
   	\mathbb{I}_{i,l} \coloneq \int_\Delta B_i(x) \, \mathrm{d}\check{\mu}(x) \, \left( X_{[p]}^\top(x) \boldsymbol{e}_l \right).
   \end{equation}
   \end{Definition}
   
      Referring to Equation \eqref{eq:perturbed measure} and \eqref{eq:I}, we have for the case of simple zeros $x_i$
   \[
   \mathbb{I}_{i,l} = \int_\Delta B_i(x) \left(
   \d\mu(x) R^{-1}(x) + \sum_{i=1}^{M} \boldsymbol{\xi}_i(x) \boldsymbol{v}_{i}^L \delta(x-x_i)\right) \d\check{\mu}(x) \, \left( X_{[p]}^\top(x) \boldsymbol{e}_l \right).
   \]
   The adjugate matrix $\operatorname{adj} R(x)$ of the matrix polynomial $R(x)$ is a matrix polynomial such that
   \[
   R^{-1}(x) = \frac{1}{\det R(x)}\operatorname{adj} R(x).
   \]
   Recalling the partial fraction decomposition
   \[
   \begin{aligned}
   	\frac{1}{\det R(x)} &= \sum_{j=1}^{M} \frac{C_j}{x-x_j}, &
   	C_j \coloneq \prod_{k \neq j} (x_j - x_k),
   \end{aligned}
   \]
   we can write, cf. \cite{Uva},
   \[
   \begin{aligned}
   	\mathbb{I}_{i,l} &= \int_\Delta B_i(x) \sum_{j=1}^{M} \left(
   	\d\mu(x) \frac{C_j}{x-x_j} \operatorname{adj} R(x)
   	+ \boldsymbol{\xi}_j(x) \boldsymbol{v}_{j}^L \delta(x-x_j)\right)  \left( X_{[p]}^\top(x) \boldsymbol{e}_l \right)\\
   	&= \sum_{j=1}^{M} 
   	C_j	\int_\Delta B_i(x) \frac{\d\mu(x)}{x-x_j}\operatorname{adj} R(x)
   	X_{[p]}^\top(x) \boldsymbol{e}_l + \sum_{j=1}^{M} B(x_j) \boldsymbol{\xi}_j(x_j) \boldsymbol{v}_{j}^L X_{[p]}^\top(x_j) \boldsymbol{e}_l.
   \end{aligned}
   \]
   This applies to the case of simple eigenvalues. It can be noted that the generalization to cases with higher multiplicities involves merely considering a more general partial fraction decomposition.
   
   \begin{Definition}
   We define new $\tilde{\tau}$-determinants as follows:
\[   \begin{aligned}
   	\tilde{\tau}_{n} &\coloneq 
   	\begin{vNiceMatrix}
   		\mathbb{I}_{0,0} & \Cdots & \mathbb{I}_{0,n} \\
   		\Vdots &        & \Vdots \\
   		\mathbb{I}_{n,0} & \Cdots & \mathbb{I}_{n,n}
   	\end{vNiceMatrix}, 
   & n &\in \{0, \dots, {Np-r}-1\}.
   \end{aligned}\]
 
   \end{Definition}

   Assuming that these determinants are nonzero, substituting Equation \eqref{EqcheckBB} into \eqref{EqOrthogo} yields:
   \begin{equation*}
   -\mathbb{I}_{n,l} = 
   \begin{bNiceMatrix}
   	\Omega_{n,0} & \Cdots & \Omega_{n,n-1}
   \end{bNiceMatrix}
   \begin{bNiceMatrix}
   	\mathbb{I}_{0,l} \\
   	\Vdots \\
   	\mathbb{I}_{n-1,l}
   \end{bNiceMatrix}.
   \end{equation*}
   Taking into account that $l \in \{0, \dots, n-1\}$, we obtain:
   \begin{equation*}
   -\begin{bNiceMatrix}
   	\mathbb{I}_{n,0} & \dots & \mathbb{I}_{n,n-1}
   \end{bNiceMatrix}
   \begin{bNiceMatrix}
   	\mathbb{I}_{0,0} & \Cdots & \mathbb{I}_{0,n-1} \\
   	\Vdots &        & \Vdots \\
   	\mathbb{I}_{n-1,0} & \Cdots & \mathbb{I}_{n-1,n-1}
   \end{bNiceMatrix}^{-1}
   = \begin{bNiceMatrix}
   	\Omega_{n,0} & \Cdots & \Omega_{n,n-1}
   \end{bNiceMatrix}.
   \end{equation*}
   
   Next, let us study the orthogonality conditions when $l = n$. In this case:
   \begin{equation*}
   \int_\Delta \check{B}_n(x) \, \mathrm{d}\check{\mu}(x) \, \left( X_{[p]}^\top(x) \boldsymbol{e}_n \right) = \alpha_n \neq 0,
   \end{equation*}
   and substituting the $\Omega$ components, we have:
   \begin{equation*}
   \mathbb{I}_{n,n} - 
   \begin{bNiceMatrix}
   	\mathbb{I}_{n,0} & \dots & \mathbb{I}_{n,n-1}
   \end{bNiceMatrix}
   \begin{bNiceMatrix}
   	\mathbb{I}_{0,0} & \Cdots & \mathbb{I}_{0,n-1} \\
   	\Vdots &        & \Vdots \\
   	\mathbb{I}_{n-1,0} & \Cdots & \mathbb{I}_{n-1,n-1}
   \end{bNiceMatrix}^{-1}
   \begin{bNiceMatrix}
   	\mathbb{I}_{0,n} \\
   	\Vdots \\
   	\mathbb{I}_{n-1,n}
   \end{bNiceMatrix} 
   = \frac{\tilde{\tau}_n}{\tilde{\tau}_{n-1}} = \alpha_n \neq 0.
   \end{equation*}
   Therefore, the orthogonality relations are satisfied provided $\tilde{\tau}_{n-1}$ and $\tilde{\tau}_n$ are nonzero.
   
   When $n = {Np-r}$, the orthogonality conditions become:
   \begin{equation*}
   -\begin{bNiceMatrix}
   	\mathbb{I}_{{Np-r},0} & \dots & \mathbb{I}_{{Np-r},{Np-r}-1}
   \end{bNiceMatrix}
   \begin{bNiceMatrix}
   	\mathbb{I}_{0,0} & \Cdots & \mathbb{I}_{0,{Np-r}-1} \\
   	\Vdots &        & \Vdots \\
   	\mathbb{I}_{{Np-r}-1,0} & \Cdots & \mathbb{I}_{{Np-r}-1,{Np-r}-1}
   \end{bNiceMatrix}^{-1}
   = \begin{bNiceMatrix}
   	\Omega_{{Np-r},0} & \Cdots & \Omega_{{Np-r},{Np-r}-1}
   \end{bNiceMatrix}.
   \end{equation*}
   
   This result, together with Proposition \ref{OmegaComponents}, establishes that the non-cancellation of $\tilde{\tau}_{{Np-r}-1}$ is equivalent to the non-cancellation of $\tau_{{Np-r}-1}$. As the $\tau$-determinants were not defined for $n \in \{0, \dots, {Np-r}-2\}$, and disregarding $\tilde{\tau}_{{Np-r}-1} \neq 0$ as a condition, we simplify the notation by dropping the $\tilde{}$ and referring to all determinants as $\tau_n$ for $n \in \mathbb{N}_0$.
   
	Theorem \ref{ExistsTheorem} extends to these $\tau$-determinants for $n \in \{0, \dots, {Np-r}-2\}$.
\begin{Proposition}
If perturbed orthogonality exists, then $\tau_n \neq 0$ for $n \in \mathbb{N}_0$.
\end{Proposition}

\begin{proof}
	Assume orthogonality exists despite $\tau_n = 0$ for some $n \in \{0, \dots, {Np-r}-2\}$. Since $\tau_n = \alpha_n \tau_{n-1}$ and $\tau_{n-1} \neq 0$, the only possibility is $\alpha_n = 0$. However,
	\begin{equation*}
		\int_\Delta \check{B}_n(x) \, \mathrm{d}\check{\mu}(x) \, \left( X_{[p]}^\top(x) \boldsymbol{e}_n \right) = \alpha_n,
	\end{equation*}
	and orthogonality would not be satisfied.
\end{proof}
	The assumption in Theorem \ref{Non-cancellation} regarding the existence of certain orthogonality conditions for $B_i(x)$ can be reformulated as the non-cancellation of the $\tau$-determinants for $n \in \{0, \dots, {Np-r}-2\}$. Consequently:
\begin{Theorem}
If $\tau_n \neq 0$ for $n \in \mathbb{N}_0$, then the perturbed orthogonality exists.
\end{Theorem}
\begin{proof}
	The assumption of non-cancellation of these determinants ensures the desired biorthogonality relations for $\check{B}_n(x)$ when $n < {Np-r}$, satisfying the conditions of Theorem \ref{Non-cancellation}.
\end{proof}

\section*{Conclusions and Outlook}

In this work, we have extended our previous study on general Christoffel transformations \cite{Manas_Rojas} to explore general Geronimus transformations for mixed multiple orthogonal polynomials. These transformations, characterized by regular matrix polynomials that are neither required to be monic nor constrained by the rank of their leading coefficients, were applied via both right and left multiplication to a matrix of measures. This approach led to Christoffel-type formulas that effectively describe how the perturbed orthogonal polynomials are related to the original ones. Importantly, we have also proven that the existence of the perturbed orthogonality is equivalent to the non-cancellation of certain $\tau$-determinants associated with the non-perturbed orthogonality.
The influence of these transformations on the Markov--Stieltjes matrix functions is also explored. As an example, we study the Jacobi--Piñeiro orthogonal polynomials with three weights.

A promising avenue for further exploration is the study of Uvarov transformations, which can be understood as a hybrid of Christoffel and Geronimus transformations, potentially offering a more comprehensive framework for perturbing orthogonality. We are also investigating the impact of these Uvarov transformations on the Markov–Stieltjes matrix functions associated with the corresponding matrix of measures.

\section*{Acknowledgments}

The authors acknowledge funding from research project [PID2021-122154NB-I00], \textit{Ortogonalidad y Aproximación con Aplicaciones en Machine Learning y Teoría de la Probabilidad}, funded by MICIU/AEI/10.13039/501100011033 and by “ERDF A Way of making Europe”. They are also grateful to JEF Díaz for pointing out the value of the Jacobi–Piñeiro polynomials at $x=1$. Finally, MM wishes to acknowledge the guidance of Francisco Marcellán in the study and research of Geronimus and Uvarov transformations.

\end{document}